\documentclass[10pt,reqno]{amsart}
\usepackage{amsmath}
\usepackage{amssymb}
\usepackage{amsthm}
\usepackage{graphicx}

\newlength{\defbaselineskip}
\newcommand{\setlinespacing}[1]%
           {\setlength{\baselineskip}{#1 \defbaselineskip}}

\numberwithin{equation}{section}

\newtheorem{thm}{Theorem}[section]
\newtheorem{cor}[thm]{Corollary}
\newtheorem{lem}[thm]{Lemma}
\newtheorem{prop}[thm]{Proposition}

\theoremstyle{definition}

\theoremstyle{remark}
\newtheorem{rem}[thm]{Remark}
\numberwithin{equation}{section}

\begin{document}

\title[Strichartz estimates for Schr\"odinger equations]
{Strichartz estimates for Schr\"odinger equations in weighted $L^2$ spaces and their applications}

\author{Youngwoo Koh and Ihyeok Seo}

\thanks{Y. Koh was supported by NRF Grant 2016R1D1A1B03932049 (Republic of Korea).
I. Seo was supported by the TJ Park Science Fellowship of POSCO TJ Park Foundation
and by the NRF grant funded by the Korea government(MSIP) (No. 2017R1C1B5017496).}

\address{Department of Mathematics Education, Kongju National University, Kongju 32588, Republic of Korea}
\email{ywkoh@kongju.ac.kr}

\address{Department of Mathematics, Sungkyunkwan University, Suwon 440-746, Republic of Korea}
\email{ihseo@skku.edu}

\maketitle


\begin{abstract}
We obtain weighted $L^2$ Strichartz estimates
for Schr\"odinger equations
$i\partial_tu+(-\Delta)^{a/2}u=F(x,t)$, $u(x,0)=f(x)$, of general orders $a>1$
with radial data $f,F$ with respect to the spatial variable $x$,
whenever the weight is in a Morrey-Campanato type class.
This is done by making use of a useful property of maximal functions of the weights
together with frequency-localized estimates which follow from using
bilinear interpolation and some estimates of Bessel functions.
As consequences, we give an affirmative answer to a question posed in \cite{BBCRV}
concerning weighted homogeneous Strichartz estimates,
and improve previously known Morawetz estimates.
We also apply the weighted $L^2$ estimates to the well-posedness theory for the
Schr\"odinger equations with time-dependent potentials in the class.
\end{abstract}

\section{Introduction}
In this paper we consider the following Cauchy problem
for Schr\"odinger equations:
    \begin{equation}\label{sch}
    \begin{cases}
    i\partial_tu+(-\Delta)^{a/2}u=F(x,t),\\
    u(x,0)=f(x),
    \end{cases}
    \end{equation}
where $(x,t)\in\mathbb{R}^{n+1}$, $n\geq2$,
and $(-\Delta)^{a/2}$ is given for $a>1$
by means of the Fourier transform $\mathcal{F}f$ $(=\widehat{f}\,)$ as follows:
$$\mathcal{F}[(-\Delta)^{a/2}f](\xi)=|\xi|^{a}\widehat{f}(\xi).$$
These equations arise in mathematical physics.
Particular interest is granted to the fractional-order cases where $1<a<2$.
This is because fractional quantum mechanics has been recently introduced
by Laskin \cite{L} where it is conjectured that physical realizations may be limited to the fractional cases.
Of course, the classical case $a=2$ has attracted interest from the ordinary quantum mechanics.
The higher-order counterpart ($a>2$) of it has been also attracted for decades
from mathematical physics.
Especially when $a=4$, \eqref{sch} can be found in the formation and propagation
of intense laser beams in a bulk medium (\cite{K,KS}).

By Duhamel's principle, we have the solution of \eqref{sch} which can be given by
\begin{equation}\label{sol0}
u(x,t)=e^{it(-\Delta)^{a/2}}f(x)-i\int_0^t e^{i(t-s)(-\Delta)^{a/2}}F(\cdot,s)ds,
\end{equation}
where the evolution operator $e^{it(-\Delta)^{a/2}}$ is defined by
$$e^{it(-\Delta)^{a/2}}f(x)=\frac1{(2\pi)^n}\int_{\mathbb{R}^n}
e^{ix\cdot\xi}e^{it|\xi|^a}\widehat{f}(\xi)d\xi.$$
There has been a lot of work on a priori estimates for the solution
which control space-time integrability of \eqref{sol0}
in view of those of the Cauchy data $f$ and $F$.
This is because they play central roles in the study of (nonlinear) dispersive equations
({\it cf. \cite{C,T}}).
In the classical case $a=2$, such an estimate was first obtained by Strichartz \cite{Str}
in $L_{t,x}^q(\mathbb{R}^{n+1})$ norms.
Since then, Strichartz's estimate has been studied by many authors
\cite{GV,KT,CW,Ka,F,V2,Ko,LS,KoS3}
naturally in more general mixed norms $L_{t}^q(\mathbb{R};L_x^r(\mathbb{R}^{n}))$.
(See also \cite{CN,S,LS2} and references therein for different related norms.)
Similar estimates are also well known for the higher-order cases
and can be found in \cite{CHL}.
In recent years, much attention has been devoted to the fractional-order cases
under the radial assumptions that the Cauchy data $f,F$ are radial with respect
to the spatial variable $x$
(see \cite{Sh,Ke,GW,CL,CKS} and references therein).

\subsection{Weighted estimates}
In this paper we address the problem of obtaining the Strichartz estimates for \eqref{sol0}
on weighted $L^2$ spaces of the form $L^2(w(x,t)dxdt)$.
More precisely, we want to find a suitable function class of weights $w(x,t)\geq0$ for which
\begin{equation}\label{homo}
\big\|e^{it(-\Delta)^{a/2}}f \big\|_{L^2(w(x,t))}\leq C_w\|f\|_{L^2}
\end{equation}
and
\begin{equation}\label{inhomo}
\bigg\|\int_{0}^{t}e^{i(t-s)(-\Delta)^{a/2}}F(\cdot,s)ds\bigg\|_{L^2(w(x,t))}
\leq C_w\|F\|_{L^2(w(x,t)^{-1})}
\end{equation}
hold. 
(For simplicity, we are using the notation $L^2(w(x,t))$ instead of $L^2(w(x,t)dxdt)$.)

When $a=2$, Strichartz estimates in a weighted $L^2$ setting as above have been studied for
time-independent weights $w(x)$ in the Morrey-Campanato class $\mathcal{L}^{\beta,p}$ defined by the norm
$$\|w\|_{\mathfrak{L}^{\beta,p}}
:=\sup_{x\in\mathbb{R}^{n},r>0}r^\beta
\bigg( \frac{1}{r^{n}} \int_{Q(x,r)} w(y)^p dy \bigg)^{1/p}<\infty$$
for $\beta>0$ and $1\leq p\leq n/\beta$ (\cite{RV2,V,BBRV,S3}).
To handle time-dependent weights in this paper, we first introduce a function class $\mathcal{L}^{\beta,p}_{a-par}$
of weights $w(x,t)$ on $\mathbb{R}^{n+1}$ which are defined by the norm
$$\|w\|_{\mathfrak{L}^{\beta,p}_{a-par}}
:=\sup_{(x,t)\in\mathbb{R}^{n+1},r>0}r^\beta
\bigg( \frac{1}{r^{n+a}} \int_{Q(x,r) \times I(t,r^a) } w(y,s)^p dyds \bigg)^{1/p}<\infty$$
for $\beta>0$, $a\geq1$ and $1\leq p\leq(n+a)/\beta$.
Here, $Q(x,r)$ denotes a cube in $\mathbb{R}^n$ centered at $x$ with side length $r$,
and $I(t,l)$ denotes an interval in $\mathbb{R}$ centered at $t$ with length $l$.
Notice that $\mathcal{L}^{\beta,p}_{a-par}$ when $a=1$ is the Morrey-Campanato class on $\mathbb{R}^{n+1}$,
and it has the homogeneity
$$\|w(\lambda\cdot,\lambda^a\cdot)\|_{\mathfrak{L}^{\beta,p}_{a-par}}
=\lambda^{-\beta}\|w\|_{\mathfrak{L}^{\beta,p}_{a-par}}.$$
This motivates the definition of the class $\mathcal{L}^{\beta,p}_{a-par}$.
In other words, it is an anisotropic variant of the usual Morrey-Campanato class adapted to scaling consideration
$(x,t)\mapsto(\lambda x,\lambda^at)$.
In this regard, when $a=2$, this class is called the parabolic Morrey-Campanato class\footnote{It was independently considered by the second author
in the study of unique continuation.} and was already appeared in \cite{BBCRV}
concerning the homogeneous estimate \eqref{homo} (see Remark \ref{rem} below).
Now we shall call $\mathcal{L}^{\beta,p}_{a-par}$ $a$-parabolic Morrey-Campanato class.
Following \cite{BBCRV}, we also observe the following properties:

\medskip

\begin{itemize}
\item
$\mathfrak{L}^{\beta,p}_{a-par}=L^p$\, when\, $p=(n+a)/\beta$,\,
and\, $L^{p,\infty}\subset\mathfrak{L}^{\beta,p}_{a-par}$\, for\, $p<(n+a)/\beta$,

\medskip

\item
$\mathfrak{L}^{\beta,p}_{a-par}\subset\mathfrak{L}^{\beta,q}_{a-par}$\, for\, $q<p$.
\end{itemize}

\medskip

In \cite{KoS2}, the estimates \eqref{hop_H_par} and \eqref{inho_par}
were obtained by the authors particularly for higher-order cases where $a>(n+2)/2$,
with a more restrictive class $\mathfrak{L}^{\alpha,\beta,p}$ of weights $w$ satisfying
$$\|w\|_{\mathfrak{L}^{\alpha,\beta,p}}
:=\sup_{(x,t)\in\mathbb{R}^{n+1},r,l>0}r^\alpha l^\beta
\bigg( \frac{1}{r^n l} \int_{Q(x,r) \times I(t,l) }w(y,s)^p dyds\bigg)^{1/p}<\infty$$
for some $0<\alpha\leq n/p$, $0<\beta\leq 1/p$ and $p\geq1$ with the scaling condition $a=\alpha+a\beta$.
(Note that $\mathfrak{L}^{\beta,p}(\mathbb{R};\mathfrak{L}^{\alpha,p}(\mathbb{R}^n))
\subset\mathfrak{L}^{\alpha,\beta,p}(\mathbb{R}^{n+1})$.)
From the definition, it is easy to check that
when $l=r^a$,
$\mathfrak{L}^{\alpha,\beta,p}$ becomes equivalent to the $a$-parabolic Morrey-Campanato class
$\mathfrak{L}^{\alpha+a\beta,p}_{a-par}$.
Hence, $\mathfrak{L}^{\alpha,\beta,p}\subset\mathfrak{L}^{a,p}_{a-par}$
under the condition $a=\alpha+a\beta$.

Our result on the estimates \eqref{homo} and \eqref{inhomo}
deals with weights in $a$-parabolic Morrey-Campanato classes,
which are the most natural Morrey-Campanato type classes adapted to scaling structure of Schr\"odinger equations,
and is stated as follows:

\begin{thm}\label{thm_par}
Let $n\geq2$ and $w\in\mathfrak{L}^{a,p}_{a-par}$.
Assume that $f$ and $F$ are radial functions with respect to the spatial variable $x$.
Then we have
\begin{equation}\label{hop_H_par}
\big\|e^{it(-\Delta)^{a/2}}f \big\|_{L^2(w(x,t))}\leq C\|w\|_{\mathfrak{L}^{a,p}_{a-par}}^{1/2}\|f\|_{L^2}
\end{equation}
and
\begin{equation}\label{inho_par}
\bigg\|\int_{0}^{t}e^{i(t-s)(-\Delta)^{a/2}}F(\cdot,s)ds\bigg\|_{L^2(w(x,t))}
\leq C\|w\|_{\mathfrak{L}^{a,p}_{a-par} }\|F\|_{L^2(w(x,t)^{-1})}
\end{equation}
if\, $a>1$ and $a/(a-1)<p\leq(n+a)/a$.
\end{thm}

\begin{rem}
Here we are assuming $\beta=a$ but this is just needed for the scaling invariance of the estimates  \eqref{hop_H_par} and \eqref{inho_par} under the scaling $(x,t)\rightarrow(\lambda x,\lambda^{a}t)$, $\lambda>0$.
\end{rem}

\begin{rem}\label{rem}
For \eqref{hop_H_par}, we will prove more generally
\begin{equation}\label{hop_H_par2}
\big\|e^{it(-\Delta)^{a/2}}f \big\|_{L^2(w(x,t))}\leq C\|w\|_{\mathfrak{L}^{a+2s,p}_{a-par}}^{1/2}\|f\|_{\dot{H}^s}
\end{equation}
if\, $a>1$, $s>-\frac a2(1-\frac1p)$ and $\max\{\frac{a}{a-1+2s},1\}<p\leq\frac{n+a}{a+2s}$.
Thus we have a smoothing effect with a gain of $\big(\frac a2(1-\frac1p)\big)^-$ in $x$.
When $a=2$, it was shown in \cite{BBCRV} that
\eqref{hop_H_par2} holds for general functions $f\in\dot{H}^s$
if $(s,1/p)$ lies in the triangle with vertices $B,C,D$
and fails if $(s,1/p)$ lies in the triangle with vertices $A,B,F$.
See Figure \ref{figure}.
So it was naturally asked in \cite{BBCRV} whether \eqref{hop_H_par2} with $a=2$ might hold
for the quadrangle with vertices $B,D,E,F$.
With the radial assumption on $f$, we give an affirmative answer to this question
that it can hold on the quadrangle and even on a region off the line $BF$.
This improvement particularly on $s=0$ enables us to apply weighted estimates like \eqref{hop_H_par2}
to the Cauchy problem \eqref{higher} with $L^2$ initial data in a weighted $L^2$ setting.
\end{rem}


\begin{figure}[t!]\label{figure}
\includegraphics[width=8.0cm]{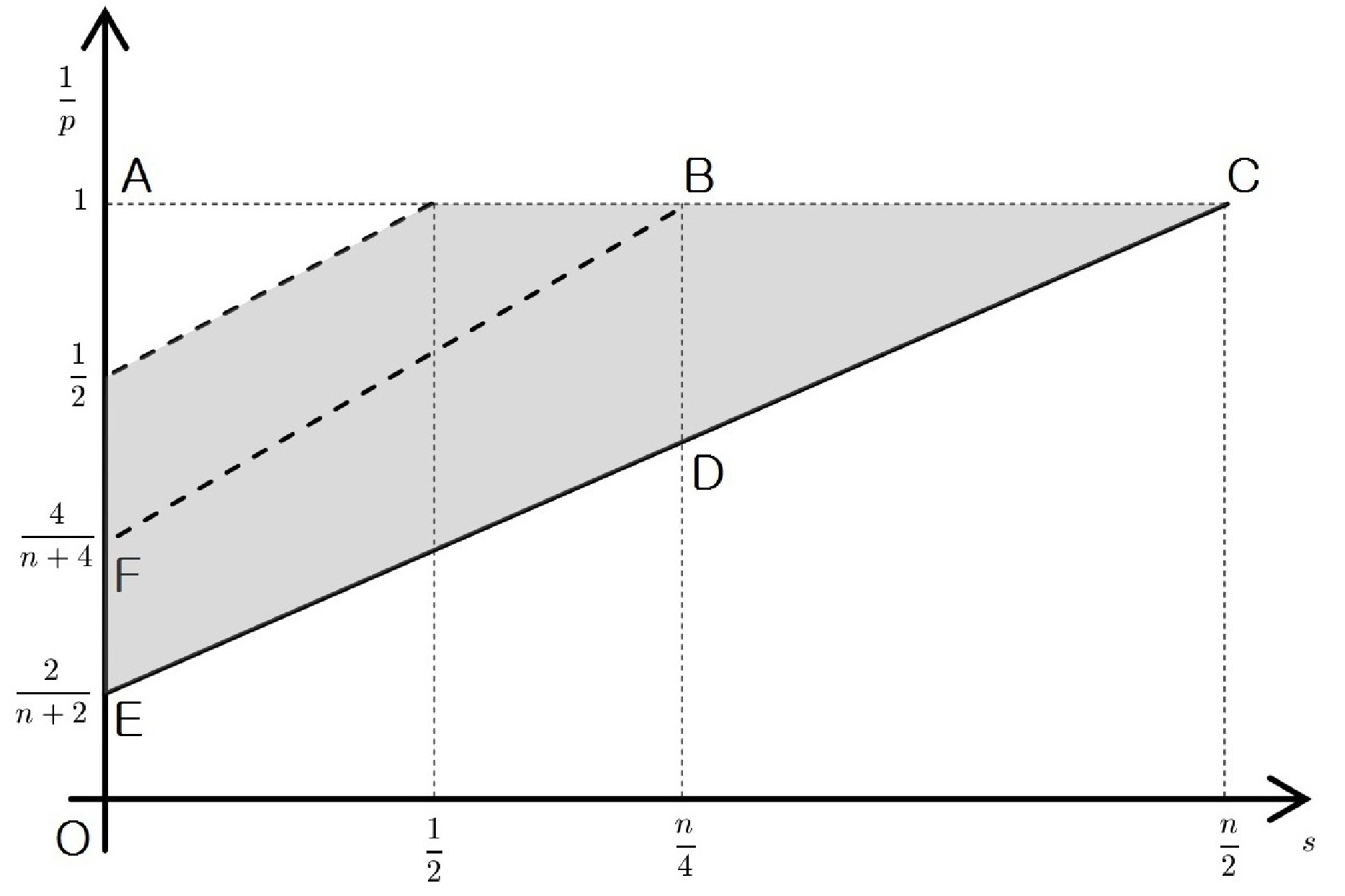}
\caption{The region of $(s,1/p)$ for \eqref{hop_H_par2}
particularly when $a=2$.}
\end{figure}


The estimate \eqref{hop_H_par2} was shown for the wave equation ($a=1$)
(\cite{KoS})
and can be compared with the following smoothing estimates (known for Morawetz estimates)
\begin{equation}\label{mor'}
\big\||x|^{-b/2}e^{itD^a}f\big\|_{L_{t,x}^2}\leq C\|D^{(b-a)/2}f\|_2
\end{equation}
which have been studied by many authors
for the wave equation ($a=1$) \cite{M} and for the Schr\"odinger equation ($a=2$) \cite{KY,Su,V}.
For fractional Schr\"odinger equations, \eqref{mor'} can be found in Theorem 1.10 of \cite{FW}
that \eqref{mor'} holds for $b \in(1,n)$ and $a>0$ if $f$ is a radial function.
As a direct consequence of \eqref{hop_H_par2},
we improve this result to cases where $b>a/p$ and extend \eqref{mor'} to more general time-dependent weights instead of $|x|^{-b}$.
Indeed, by taking $s=(b-a)/2$ in \eqref{hop_H_par2} and using the simple relation
$\|w^b\|_{\mathfrak{L}^{b,p}_{a-par}} =\|w\|_{\mathfrak{L}^{1,pb}_{a-par}}^b$,
we get the following corollary.

\begin{cor}
Let $n\geq2$, $a>1$, $b\in(a/p,n+a)$ and $f$ be a radial function. Then we have
\begin{equation}\label{coo}
\big\||w(x,t)|^{b/2}e^{itD^a}f\big\|_{L_{t,x}^2}
\leq C\|D^{(b-a)/2}f\|_2
\end{equation}
if $w\in\mathfrak{L}^{1,pb}_{a-par}(\mathbb{R}^{n+1})$ for $\max\{a/(b-1),1\}<p\leq (n+a)/b$.
\end{cor}

\begin{rem}
Since $|x|^{-\alpha} |t|^{-\beta} \in\mathfrak{L}^{1,pb}_{a-par}(\mathbb{R}^{n+1})$ for $\alpha+\beta=1$ and $\alpha,\beta\geq0$,
\eqref{coo} covers particularly \eqref{mor'}.
\end{rem}

\subsection{Applications}
Now we present a few applications of our estimates to the well-posedness theory
for the following Cauchy problem in the radial case:
    \begin{equation}\label{higher}
    \begin{cases}
    i\partial_tu+(-\Delta)^{a/2}u+V(x,t)u=F(x,t),\quad a>1,\\
    u(x,0)=u_0(x),
    \end{cases}
    \end{equation}
where we assume that $u,u_0,V$ and $F$ are
radial functions with respect to the spatial variable $x$.
Making use of Theorem \ref{thm_par}, we obtain here that \eqref{higher} is globally well-posed
in the space $L^2(|V|dxdt)$ with potentials $V\in\mathfrak{L}^{a,p}_{a-par}$.
More precisely, we have the following result.

\begin{thm}\label{thm0}
Let $n\geq2$, $a >1$ and $a/(a-1)<p\leq(n+a)/a$.
Assume that $V\in\mathfrak{L}^{a,p}_{a-par}$ with
$\|V\|_{\mathfrak{L}^{a,p}_{a-par}}$ small enough,
and that $u,u_0,V$ and $F$ are
radial functions with respect to the spatial variable $x$.
Then, if $u_0\in L^2$ and $F\in L^2(|V|^{-1})$, there exists a unique solution of the problem \eqref{higher}
in the space $L^2(|V|)$. Furthermore, the solution $u$ belongs to $C_tL_x^2$ and satisfies the following inequalities:
\begin{equation}\label{1-1}
\|u\|_{L^2(|V|)}\leq C\|V\|_{\mathfrak{L}^{a,p}_{a-par}}^{1/2}\|u_0\|_{L^2}+C\|V\|_{\mathfrak{L}^{a,p}_{a-par}}\|F\|_{L_{t,x}^2(|V|^{-1})}
\end{equation}
and
\begin{equation}\label{1-2}
\sup_{t\in\mathbb{R}}\|u\|_{L_x^2}\leq C\|u_0\|_{L^2}+C\|V\|_{\mathfrak{L}^{a,p}_{a-par}}^{1/2}\|F\|_{L_{t,x}^2(|V|^{-1})}.
\end{equation}
\end{thm}

The well-posedness for linear Schr\"odinger equations with potentials has been studied by many authors (see \cite{RV,RV2,D'APV,NS,BBRV,S3,KoS2,CKS}).
In the context of the weighted $L^2$ setting, it has been studied
in \cite{RV2,V,BBRV,S3} essentially for time-independent potentials $V(x)$ contained in various classes like Morrey-Campanato classes.
The first result on time-dependent potentials was obtained in our previous work \cite{KoS2} where
we prove Theorem \ref{thm0} for higher orders $a>(n+2)/2$ and a more restrictive potential class $\mathfrak{L}^{\alpha,\beta,p}$
with $a=\alpha+a\beta$.
The contribution of Theorem \ref{thm0} is that we allow general orders $a>1$ and  $a$-parabolic Morrey-Campanato potential classes
which are the most natural Morrey-Campanato type classes adapted to scaling structure of Schr\"odinger equations.
See also \cite{KoS} for a related result where we consider the wave equation corresponding to the case $a=1$.

\subsection{Main ideas}
We end this section with an outline of the main ideas and the organization of this paper.

The known approach to weighted $L^2$ Strichartz estimates with time-independent weights is based on weighted $L^2$ bounds for resolvent of the Laplacian and for Fourier restriction (see, for example, \cite{RV2,V,BBRV,S3}),
but it is no longer available in the case of general time-dependent weights.
Our method that works for the time-dependent case is entirely different from them
and is based on a combination of the usual $TT^*$ argument, bilinear interpolation, and frequency as well as spatial localization
based on a property of maximal functions of weights and asymptotic expansion of Bessel functions.
This is done in several steps:

\smallskip

\noindent\textit{Frequency localization on weighted spaces and maximal functions of weights.}
To show \eqref{hop_H_par2} as well as the inhomogeneous estimate, the first step is to work on spatial Fourier transform side by making use of the Littlewood-Paley theorem on weighted $L^2$ spaces
with Muckenhoupt $A_2$ weights in the spatial variable.
A key observation in this step is to remove the $A_2$ assumption $w(\cdot,t)\in A_2(\mathbb{R}^n)$ when applying the theorem
by using a useful property (Lemma \ref{lem2_par}) of $n$-dimensional maximal functions
$w_*(x,t)=(M(w(\cdot,t)^q)(x))^{1/q}$ of weights in $a$-parabolic Morrey-Campanato classes.
Here, $M(f)$ is the usual Hardy-Littlewood maximal function of $f$.
This property, which is the main part of Section \ref{sec2}, says that
$\|w_*\|_{\mathfrak{L}^{\beta,p}_{a-par}}\leq C\|w\|_{\mathfrak{L}^{\beta,p}_{a-par}}$ if $\beta>a/p$ and $p>q$,
and $w_\ast(\cdot,t)\in A_2(\mathbb{R}^n)$ uniformly in $t\in\mathbb{R}$.
Now, since $w\leq w_\ast$ and $\|w_*\|_{\mathfrak{L}^{\beta,p}_{a-par}}\leq C\|w\|_{\mathfrak{L}^{\beta,p}_{a-par}}$,
it suffices to show \eqref{hop_H_par2} replacing $w$ with $w_\ast$.
So we may assume $w(\cdot,t)\in A_2(\mathbb{R}^n)$ for simplicity and can now apply the Littlewood-Paley theorem to \eqref{hop_H_par2}
without assuming the $A_2$ condition on the weight.
This finally leads us to estimating a number of frequency-localized pieces like
    \begin{equation}\label{freq_par00}
    \big\|e^{it(-\Delta)^{a/2}} P_k f \big\|_{L^2(w)}\leq C2^{k(\beta -a)/2}
    \|w\|_{\mathfrak{L}^{\beta,p}_{a-par}}^{1/2}\|f\|_{L^2}
    \end{equation}
in Proposition \ref{prop_par}, where $\beta>1+a/p$, $a>1$, $1<p\leq(n+a)/\beta$ and $P_k$ is the Littlewood-Paley projection.
This approach allows us to take advantage of localization in Fourier transform side which is a basic strategy in our argument.
See Section \ref{sec3} for details.

\smallskip

\noindent\textit{$TT^*$ argument and asymptotic expansion of Bessel functions.}
The next step is devoted to proving frequency-localized estimates like \eqref{freq_par00}
whose proof is based on a combination of the usual $TT^*$ argument, spatial localization argument based on asymptotic expansion (Lemma \ref{bessel})
of Bessel functions, and bilinear interpolation (Lemmas \ref{inter} and \ref{id}).

We shall give here a brief description of this step. See Section \ref{sec4} for details.
Notice first that we only need to show the case $k=0$ in \eqref{freq_par00} by the scaling $(x,t)\rightarrow(\lambda x,\lambda^{a}t)$.
Then, by using the usual $TT^*$ argument we are reduced to showing the bilinear form estimate \eqref{bfs00},
    \begin{equation*}
    \bigg|\bigg\langle\int_{\mathbb{R}} \Big( e^{i(t-s)(-\Delta)^{a/2}}P_0^2F(\cdot,s) \Big)(x) ds,G(x,t)\bigg\rangle\bigg|
    \leq C\|w\|_{\mathfrak{L}^{\beta,p}_{a-par}}\|F\|_{L^2(w^{-1})}\|G\|_{L^2(w^{-1})}.
    \end{equation*}
Of course, $F$ and $G$ are assumed here to be radial with respect to the space variable $x$.
By decomposing dyadically the involved functions $F,G$ into spatially localized pieces $F_j,G_k$,
we are reduced to estimating a number of spatially localized pieces like
\begin{align}\label{bilin_par00}
\nonumber\bigg|\bigg\langle\int_{\mathbb{R}} \Big( e^{i(t-s)(-\Delta)^{a/2}}P_0^2F_j(\cdot,s) \Big)(x) ds, &G_k(x,t)\bigg\rangle\bigg|\\
\leq C2^{C(j,k,\beta,p,a)}&\|w\|_{\mathfrak{L}^{\beta,p}_{a-par}}\|F\|_{L^2(w^{-1})}\|G\|_{L^2(w^{-1})}
\end{align}
with suitable constants $C(j,k,\beta,p,a)$.
Then using the Fourier transform of spherical surface measure \eqref{spheri},
the integral in the right-hand side of \eqref{bilin_par00} can be written in terms of Bessel functions $J_{\frac{n-2}{2}}$,
as follows (see \eqref{0907}):
    \begin{align}\label{22est00}
    &\int_{\mathbb{R}} \Big(e^{i(t-s) (-\Delta)^{a/2}} P_0^2 F_j(\cdot,s)\Big)(x) ds \\
    \nonumber=&r^{-\frac{n-2}{2}} \int_{\mathbb{R}}\int_{0}^\infty \lambda^{-\frac{n-2}{2}}
        \bigg( \int_{0}^\infty e^{i(t-s)\rho^a} \overline{J_{\frac{n-2}{2}} (r\rho)} J_{\frac{n-2}{2}} (\lambda\rho) \rho \psi(\rho)^2 d\rho \bigg)
        \lambda^{n-1}\widetilde{F_j}(\lambda,s) d\lambda ds.
    \end{align}
Note here that $\widetilde{F_j}(\lambda,s):=F_j(\lambda y',s)$ is independent of $y'\in S^{n-1}$
since $F_j$ is a radial function in the $x$ variable.
The radial assumption comes into play at this step.

As mentioned above, the Morawetz estimates \eqref{mor'} and the classical $L^q-L^r$ Strichartz estimates
have been studied under the radial assumptions on Cauchy data in the fractional case $1<a<2$,
and its method is based on asymptotic expansion of Bessel functions similar as in our case
(see \cite{Sh,Ke,GW,CL,CKS} and references therein).
Up to now, such an approach seems quite general when handling the fractional case.
We think that this is because such estimates are due to the dispersive nature of the equation,
but the dispersion in the fractional case seems not to be strong enough to have the estimates under general data.
This naturally leads us to consider the possibility of having the estimates under radial data.
Following this approach, we assume radial symmetry on the data, although a big picture in our argument is not restricted to
the case having the symmetry.
Of course, there is still a possibility in higher orders $a\geq2$ that we have such weighted estimates for general data,
as shown in our previous work \cite{KoS2}.

\smallskip

\noindent\textit{Bilinear interpolation .}
Using \eqref{22est00} and the asymptotic expansion of Bessel functions, we get
\eqref{bilin_par00} by dividing cases into $j,k\geq0$, $|j-k|\leq1$ and $|j-k|>1$.
See \eqref{inter_1}, \eqref{inter_3'} and \eqref{inter_3}, respectively,
which are proved through Subsubsection \ref{sec4.1.1} and Subsection \ref{sec4.2}.
To sum \eqref{bilin_par00} over $j,k\geq0$ in Subsection \ref{sec4.1},
we finally interpolate these three cases using the bilinear interpolation lemmas \ref{inter} and \ref{id}.
The corresponding inhomogeneous part is similarly handled in Subsection \ref{sec4.3}.

\medskip

In the final section, Section \ref{sec5}, we make use of our weighted $L^2$ Strichartz estimates
to obtain the global well-posedness result, Theorem \ref{thm0}.

\

Throughout this paper, we will use the letter $C$ to denote positive constants
which may be different at each occurrence.
We also denote $A\lesssim B$ and $A\sim B$ to mean $A\leq CB$ and $CB\leq A\leq CB$, respectively,
with unspecified constants $C>0$.


\section{Preliminary lemmas}\label{sec2}

In this section we present preliminary lemmas which will be used in later sections
for the proof of Theorem \ref{thm_par}.

Let us first recall that a weight\footnote{\,It is a locally integrable function
which is allowed to be zero or infinite only on a set of Lebesgue measure zero.}
$w:\mathbb{R}^n\rightarrow[0,\infty]$
is said to be in the Muckenhoupt $A_2(\mathbb{R}^n)$ class if there is a constant $C_{A_2}$ such that
\begin{equation*}
\sup_{Q\text{ cubes in }\mathbb{R}^{n}}
\bigg(\frac1{|Q|}\int_{Q}w(x)dx\bigg)\bigg(\frac1{|Q|}\int_{Q}w(x)^{-1}dx\bigg)<C_{A_2}.
\end{equation*}
(See, for example, \cite{G2}.)
We also say that $w$ is in the class $A_1$ if there is a constant $C_{A_1}$
such that for almost every $x$
\begin{equation*}
M(w)(x)\leq C_{A_1}w(x),
\end{equation*}
where $M(w)$ is the Hardy-Littlewood maximal function of $w$ defined by
\begin{equation}\label{maxi}
M(w)(x)=\sup_{Q}\frac1{|Q|}\int_{Q}w(y)dy.
\end{equation}
Here, the sup is taken over all cubes $Q$ in $\mathbb{R}^{n}$ with center $x$.
Then,
\begin{equation}\label{bac}
A_1\subset A_2\quad\text{with}\quad C_{A_2}\leq C_{A_1}.
\end{equation}
(See \cite{G2} for details.)
In the following lemma, we give a useful property of weights in $a$-parabolic Morrey-Campanato classes
regarding the maximal function $w_*(x,t)=(M(w(\cdot,t)^q)(x))^{1/q}$.
Similar properties for Morrey-Campanato type classes can be also found in \cite{CS,KoS,KoS2}.
Such property has been studied earlier in \cite{CS,S2,S3} concerning unique continuation for Schr\"odinger equations.

\begin{lem}\label{lem2_par}
Let $w\in\mathfrak{L}^{\beta,p}_{a-par}$ be a weight on $\mathbb{R}^{n+1}$
and $w_*(x,t)$ be the $n$-dimensional maximal function defined by
$$w_*(x,t)=\sup_{Q'}\Big(\frac{1}{|Q'|}\int_{Q'}w(y,t)^q dy\Big)^{\frac{1}{q}},\quad q>1,$$
where $Q'$ denotes a cube in $\mathbb{R}^n$ with center $x$.
Then, if $\beta>a/p$ and $p>q$, we have
$\|w_*\|_{\mathfrak{L}^{\beta,p}_{a-par}}\leq C\|w\|_{\mathfrak{L}^{\beta,p}_{a-par}}$,
and $w_\ast(\cdot,t)\in A_2(\mathbb{R}^n)$ in the $x$ variable with a constant $C_{A_2}$ uniform in almost every $t\in\mathbb{R}$.
\end{lem}

\begin{proof}
We first show that $\|w_*\|_{\mathfrak{L}^{\beta,p}_{a-par}}\leq C\|w\|_{\mathfrak{L}^{\beta,p}_{a-par}}$.
Fix a cube $Q(z,r)\times I(\tau,r^a)$ in $\mathbb{R}^{n+1}$.
Here, $Q(z,r)$ denotes a cube in $\mathbb{R}^n$ centered at $z$ with side length $r$,
and $I(\tau,r^a)$ denotes an interval in $\mathbb{R}$ centered at $\tau$ with length $r^a$.
Then, we define the rectangles $R_k$, $k\geq1$, such that $(y,t)\in R_k$ if $|t-\tau|<2r^a$
and $y\in Q(z,2^{k+1}r)\setminus Q(z,2^kr)$,
and set $R_0=Q(z,2r)\times I(\tau,4r^a)$.

Now one may write
$$w(y,t)=\sum_{k\geq 0}w^{(k)}(y,t)+\phi(y,t),$$
where $w^{(k)}=w\chi_{R_k}$ with the characteristic function $\chi_{R_k}$ of the set $R_k$,
and $\phi(y,t)$ is a function supported on $\mathbb{R}^{n+1}\setminus\bigcup_{k\geq0}R_k$.
Also it is easy to see that
$$w_*(x,t)\leq\sum_{k \geq 0}\big(w^{(k)}\big)_*(x,t)+\phi_*(x,t)$$
and by Minkowski's inequality
\begin{align}\label{lab}
\nonumber\Big(\int_{Q(z,r)\times I(\tau,r^a)}w_*(x,t)^pdxdt\Big)^{\frac{1}{p}}
&\leq\sum_{k \geq 0}\Big(\int_{Q(z,r)\times I(\tau,r^a)} \big(w^{(k)}\big)_*(x,t)^p dxdt \Big)^{\frac{1}{p}}\\
&+\Big( \int_{Q(z,r)\times I(\tau,r^a)}\phi_*(x,t)^pdxdt\Big)^{\frac{1}{p}}.
\end{align}
Since $\phi_*(x,t)$ is a maximal function of $\phi(y,t)$ with respect to the spatial variable $y$,
and $\phi(y,t)=0$ if $|t-\tau|<2r^a$ (from the support of $\phi$),
we see that $\phi_*(x,t)=0$ if $(x,t)\in Q(z,r) \times I(\tau,r^a)$.
Hence we may consider only the first part in the right-hand side of \eqref{lab}.

For the term where $k=0$, we use the following well-known maximal theorem
\begin{equation}\label{maxi0}
\|M(f)\|_s\leq C\|f\|_s,\quad s>1,
\end{equation}
where $M(f)$ is the Hardy-Littlewood maximal function defined as in \eqref{maxi}.
Indeed, by applying \eqref{maxi0} with $s=p/q$ in $x$-variable,
we see that if $p>q$
\begin{align}\label{sso}
\nonumber r^\beta \Big( \frac{1}{r^{n+a}} \int_{Q(z,r) \times I(\tau,r^a)}&\big(w^{(0)}\big)_* (x,t)^p dxdt \Big)^{\frac{1}{p}}\\
\nonumber\leq &C r^\beta \Big( \frac{1}{r^{n+a}} \int_{Q(z,2r) \times I(\tau,4r^a)} w (y,t)^p dydt \Big)^{\frac{1}{p}}\\
\leq &C\|w\|_{\mathfrak{L}^{\beta,p}_{a-par}}.
\end{align}

Now we only need to consider the terms where $k\geq1$.
Let $k \geq 1$.
Since $(x,t)\in Q(z,r) \times I(\tau,r^a)$, it follows that
    \begin{align*}
   \big(w^{(k)}\big)_*(x,t)
    &= \sup_{Q' \subset \mathbb{R}^n}\bigg(\frac{1}{|Q'|}\int_{Q'} w(y,t)^q\chi_{R_k}(y,t)dy\bigg)^{\frac{1}{q}}\\
    &\leq C\bigg(\frac{1}{(2^k r)^n}\int_{Q(z,2^{k+1} r)\setminus Q(z, 2^k r)} w(y,t)^q dy \bigg)^{\frac{1}{q}}\\
    &\leq C\bigg(\frac{1}{(2^k r)^n}
    \int_{Q(z,2^{k+1}r)\setminus Q(z, 2^k r)} w(y,t)^pdy\bigg)^{\frac{1}{p}},
    \end{align*}
where, for the first inequality  we used the fact that $w(y,t)^q\chi_{R_k}\neq0$ only if
$y\in Q'$ such that $|Q'|\geq(2^kr)^n-r^n\geq\frac12(2^kr)^n$,
and for the last inequality we used H\"older's inequality since $p\geq q$.
Hence,
    \begin{align*}
    \int_{Q(z,r) \times I(\tau,r^a)}&\big(w^{(k)}\big)_*(x,t)^pdxdt \\
    \leq &\frac{C}{(2^kr)^n}\int_{|\tau-t|< r^a}
    \int_{Q(z,2^{k+1}r)\setminus Q(z, 2^k r)}w(y,t)^p\int_{|z-x|<r}1\,dxdydt\\
    \leq&\frac{C}{2^{kn}}\int_{R_k}w(y,t)^pdydt.
    \end{align*}
Since $R_k\subset Q(z,2^{k+1}r)\times I(\tau,2r^a)$, this implies that
    \begin{align*}
    r^\beta \Big( \frac{1}{r^{n+a}}& \int_{Q(z,r) \times I(\tau,r^a)} \big(w^{(k)}\big)_* (x,t)^p dxdt \Big)^{\frac{1}{p}} \\
    &\leq C r^\beta \Big( \frac{1}{2^{kn} r^{n+a}}\int_{R_k} w(y,t)^pdydt \Big)^{\frac{1}{p}} \\
    &\leq C 2^{-\beta k} (2^k r)^\beta \Big( \frac{1}{2^{-ka} (2^{k}r)^{n+a}} \int_{Q(z,2^{k+1}r)\times I(\tau,2r^a)} w(y,t)^p dy dt \Big)^{\frac{1}{p}} \\
    &\leq C 2^{-\beta k} (2^k r)^\beta \Big( \frac{1}{2^{-ka} (2^{k}r)^{n+a}} \int_{Q(z,2^{k+1}r)\times I(\tau,(2^{k+1}r)^a)} w(y,t)^p dy dt \Big)^{\frac{1}{p}} \\
    &\leq C 2^{-\beta k + \frac{ak}{p}} \|w\|_{\mathfrak{L}^{\beta,p}_{a-par}}.
    \end{align*}
Hence, since $\beta>a/p$ and $p\geq q$, it follows that
    $$\sum_{k \geq 1}r^\beta \Big( \frac{1}{r^{n+a}} \int_{Q(z,r) \times Q(\tau,r^a)} \big(w^{(k)}\big)_*(x,t)^pdxdt\Big)^{\frac{1}{p}}\leq C\|w\|_{\mathfrak{L}^{\beta,p}_{a-par}}.$$
By combining this and \eqref{sso}, we get
$\|w_*\|_{\mathfrak{L}^{\beta,p}_{a-par}}\leq C\|w\|_{\mathfrak{L}^{\beta,p}_{a-par}}$
if $\beta>a/p$ and $p>q$.

\medskip

It remains to show that $w_\ast(\cdot,t)\in A_2(\mathbb{R}^n)$.
For this, we will make use of the following fact that can be found in Chapter 5 of \cite{St}
(see also Proposition 2 in \cite{CR}):
If $M(w)(x)<\infty$ for almost every $x\in\mathbb{R}^n$, then for every $\delta\in(0,1)$
\begin{equation}\label{max}
(M(w))^\delta\in A_1
\end{equation}
with $C_{A_1}$ independent of $w$.
Now we are ready to show that
$w_\ast(\cdot,t)\in A_2(\mathbb{R}^n)$ in the $x$ variable with a constant $C_{A_2}$ uniform in almost every $t\in\mathbb{R}$.
Note first that
$$w_\ast(x,t)=(M(w(\cdot,t)^q))^{1/q}.$$
Since $w\in\mathfrak{L}_{a-par}^{\beta,p}$ and $p\geq q$, it is not difficult to see that
$M(w(\cdot,t)^q)<\infty$ for almost every $x\in\mathbb{R}^n$.
Then, by applying \eqref{max} with $\delta=1/q$,
we see that $w_\ast(\cdot,t)\in A_1$ with $C_{A_1}$ uniform in $t\in\mathbb{R}$.
Finally, from \eqref{bac}, this implies immediately that
$w_\ast(\cdot,t)\in A_2$ with $C_{A_2}$ uniform in $t\in\mathbb{R}$.
\end{proof}

Let $\{A_0,A_1\}$ be an interpolation couple.
Namely, $A_0$ and $A_1$ are two complex Banach spaces, both linearly and continuously
embedded in a linear complex Hausdorff space.
For $0<t<\infty$ and $a\in A_0+A_1$, let us set
$$K(t,a)=\inf_{a=a_0+a_1}\|a_0\|_{A_0}+t\|a_1\|_{A_1}.$$
For $0<\theta<1$ and $1\leq q\leq\infty$, we denote by $(A_0,A_1)_{\theta,q}$
the real interpolation spaces equipped with the norms
$\|a\|_{(A_0,A_1)_{\theta,\infty}}=\sup_{0<t<\infty}t^{-\theta}K(t,a)$
and $\|a\|_{(A_0,A_1)_{\theta,q}}=\big(\int_0^\infty(t^{-\theta}K(t,a))^q\big)^{1/q}$, $1\leq q<\infty$.
In particular, $(A_0,A_1)_{\theta,q}=A_0=A_1$ if $A_0=A_1$.
See ~\cite{BL,Tr} for details.

We recall here two existing results concerning the real interpolation spaces.
The first one is the following bilinear interpolation lemma (see \cite{BL}, Section 3.13, Exercise 5(a)).

\begin{lem}\label{inter}
For $i=0,1$, let $A_i,B_i,C_i$ be Banach spaces and let $T$ be a bilinear operator such that
$$T:A_0\times B_0\rightarrow C_0
\quad\text{and}\quad
T:A_1\times B_1\rightarrow C_1.$$
Then one has
$$T:(A_0,A_1)_{\theta,p_1}\times(B_0,B_1)_{\theta,p_2}\rightarrow(C_0,C_1)_{\theta,q}$$
if\, $0<\theta<1$, $1\leq q\leq\infty$ and $1/q=1/p_1+1/p_2-1$.
\end{lem}

For $s\in\mathbb{R}$ and $1\leq q\leq\infty$,
let $\ell^s_q$ denote the weighted sequence space with the norm
$$\|\{x_j\}_{j\geq0} \|_{\ell^s_q}=
\begin{cases}
\big(\sum_{j\geq0}2^{jsq}|x_j|^q\big)^{1/q}
\quad\text{if}\quad q\neq\infty,\\
\,\sup_{j\geq0}2^{js}|x_j|
\quad\text{if}\quad q=\infty.
\end{cases}$$
Then the second one concerns some useful identities of real interpolation spaces of weighted spaces
(see Theorems 5.4.1 and 5.6.1 in \cite{BL}):

\begin{lem}\label{id}
Let $0<\theta<1$. Then one has
$$( L^{2} (w_0), L^{2} (w_1) )_{\theta,2} = L^2(w),\quad w= w_0^{1-\theta} w_1^{\theta},$$
and for $1\leq q_0,q_1,q\leq\infty$ and $s_0\neq s_1$,
$$(\ell^{s_0}_{q_0}, \ell^{s_1}_{q_1} )_{\theta, q}=\ell^s_q,\quad s= (1-\theta)s_0 + \theta s_1.$$
\end{lem}

The following lemma can be seen as a version of the van der Corput lemma (\cite{St}, Chap. VIII)
to suit our purpose, and will be used in Subsection \ref{sec4.2} for the proof of Lemma \ref{local_prop}.

\begin{lem}\label{VanCor+IntPart}
Let $a>1$ and $t\in \mathbb{R}$. Then,
\begin{equation}\label{vand}
\bigg| \int e^{\pm iR\rho +it\rho^a} \Phi(\rho) d\rho \bigg|
\lesssim R^{-\frac{1}{2}}\big(\|\Phi\|_{L^\infty}+\|\Phi'\|_{L^\infty}\big)
\end{equation}
for $R>1$ and $\Phi\in C^1(\mathbb{R})$ supported in $[1/2,2]$.
Here, $\Phi'$ denotes the derivative $d\Phi/d\rho$.
\end{lem}

\begin{proof}
We first decompose the left-hand side of \eqref{vand} as
    \begin{align*}
    \bigg| \int e^{\pm iR\rho +it\rho^a} \Phi(\rho) d\rho \bigg|
    &\leq \chi_{\{\frac{R}{8a}<|t|<\frac{8R}a\}}(t) \bigg| \int e^{\pm iR\rho +it\rho^a} \Phi(\rho) d\rho \bigg| \\
    &\quad + (1-\chi_{\{\frac{R}{8a}<|t|<\frac{8R}a\}}(t))
    \bigg| \int e^{\pm iR\rho +it\rho^a} \Phi(\rho) d\rho \bigg|.
    \end{align*}
Then, when $\frac R{8a}<|t|<\frac{8R}{a}$, by the van der Corput lemma,
it follows that
    $$
    \bigg| \int e^{\pm iR\rho +it\rho^a} \Phi(\rho) d\rho \bigg|
    \lesssim R^{-\frac{1}{2}}\big(\|\Phi\|_{L^\infty}+\|\Phi'\|_{L^\infty}\big).
    $$
For the second part where $|t|>\frac{8R}a$ or $|t|<\frac R{8a}$,
we first see that
    \begin{align*}
    &\int \Big(1-\frac{a(a-1)t\rho^{a-2}}{i(\pm R +a t\rho^{a-1})^2}\Big)
        e^{\pm iR\rho+it\rho^a} \Phi(\rho) d\rho\\
    &\quad=\Big[\frac{1}{i(\pm R +a t\rho^{a-1})}e^{\pm iR\rho+it\rho^a}
        \Phi(\rho)  \Big]_{\rho=1/2}^{\rho=2}
    -\int \frac{1}{i(\pm R+at\rho^{a-1})}e^{\pm iR\rho+it\rho^a}\Phi'(\rho)d\rho
    \end{align*}
by the integration by parts.
Here we note that $|\pm R+a t\rho^{a-1}|\gtrsim R$
when $|t|>\frac{8R}a$ or $|t|<\frac R{8a}$.
From this and the support of $\Phi$, we now get
\begin{align*}
\bigg|\int e^{\pm iR\rho+it\rho^a} \Phi(\rho) d\rho\bigg|
&\leq\int\bigg|\frac{a(a-1)t\rho^{a-2}}
{(\pm R +a t\rho^{a-1})^2} \Phi(\rho) \bigg|d\rho
+\int\bigg|\frac{1}{(\pm R +a t\rho^{a-1})}\Phi'(\rho) \bigg| d\rho\\
&\lesssim R^{-1} \big( \| \Phi \|_{L^\infty} + \|\Phi'\|_{L^\infty} \big)
\end{align*}
as desired.
\end{proof}


\section{Proof of Theorem \ref{thm_par}}\label{sec3}

This section is devoted to proving Theorem \ref{thm_par} assuming Proposition \ref{prop_par}
which will be proved in Section \ref{sec4}.

Let us first consider the multiplier operators $P_kf$ for $k\in\mathbb{Z}$ which are defined by
$$\widehat{P_kf}(\xi)=\psi(2^{-k}|\xi|)\widehat{f}(\xi),$$
where $\psi:\mathbb{R}\rightarrow[0,1]$ is a smooth cut-off function which is supported in $(1/2,2)$
and satisfies
$$\sum_{k=-\infty}^\infty \psi(2^{-k}t)=1,\quad t>0.$$
Then we will obtain the following frequency localized estimates in the next section
which imply Theorem \ref{thm_par}
using Lemma \ref{lem2_par} and the Littlewood-Paley theorem on weighted $L^2$ spaces.

\begin{prop}\label{prop_par}
Let $n\geq2$.
Assume that $f$ and $F$ are radial functions with respect to the spatial variable $x$.
Then we have
    \begin{equation}\label{freq_par}
    \big\|e^{it(-\Delta)^{a/2}} P_k f \big\|_{L^2(w)}\leq C2^{k(\beta -a)/2}
    \|w\|_{\mathfrak{L}^{\beta,p}_{a-par}}^{1/2}\|f\|_{L^2}
    \end{equation}
and
    \begin{equation}\label{base_par}
    \bigg\|\int_{0}^{t}e^{i(t-s)(-\Delta)^{a/2}}P_k F(\cdot,s)ds\bigg\|_{L^2(w)}
    \leq C2^{k(\beta -a)/2}\|w\|_{\mathfrak{L}^{\beta,p}_{a-par}} \|F\|_{L^2(w^{-1})}
    \end{equation}
if\, $\beta>1+a/p$, $a>1$ and $1<p\leq(n+a)/\beta$.
\end{prop}

To deduce Theorem \ref{thm_par} from this proposition, we first observe that
we may assume $w(\cdot,t)\in A_2(\mathbb{R}^n)$ uniformly in almost every $t\in\mathbb{R}$.
Indeed, since $w\leq w_\ast$ and
$\|w_*\|_{\mathfrak{L}_{a-par}^{\beta,p}}\leq C\|w\|_{\mathfrak{L}_{a-par}^{\beta,p}}$
for $\beta>a/p$ and $p>q>1$ (see Lemma \ref{lem2_par}),
if we show the homogeneous estimate \eqref{hop_H_par2} replacing $w$ with $w_\ast$, we get
    \begin{align*}
    \big\|e^{it(-\Delta)^{a/2}} f \big\|_{L^2(w(x,t))}
    &\leq \big\|e^{it(-\Delta)^{a/2}} f \big\|_{L^2(w_*(x,t))} \\
    &\leq C\|w_*\|_{\mathfrak{L}_{a-par}^{a+2s,p}}^{1/2}\|f\|_{\dot{H}^s}\\
    &\leq C\|w\|_{\mathfrak{L}_{a-par}^{a+2s,p}}^{1/2}\|f\|_{\dot{H}^s}
    \end{align*}
as desired.
Similarly for the inhomogeneous estimate \eqref{inho_par}.
So we may show the estimates \eqref{hop_H_par2} and \eqref{inho_par} by replacing $w$ with $w_\ast$.
By this replacement and the property $w_\ast(\cdot,t)\in A_2(\mathbb{R}^n)$ in Lemma \ref{lem2_par},
we may assume, for simplicity of notation, that $w(\cdot,t)\in A_2(\mathbb{R}^n)$ uniformly in almost every $t\in\mathbb{R}$.
Since the constant $C_{A_1}$ in \eqref{max} is independent of $w$,
(from the proof of Lemma \ref{lem2_par}) we see that $w_\ast(\cdot,t)\in A_2(\mathbb{R}^n)$ with $C_{A_2}$ independent of $w$.
Thus we may also assume that $w(\cdot,t)\in A_2(\mathbb{R}^n)$  with $C_{A_2}$ independent of $w$.

By this $A_2$ condition we can use the Littlewood-Paley theorem on weighted $L^2$ spaces
(see Theorem 1 in \cite{Ku} and also Theorem 5 in \cite{Kop})
to get
    \begin{align*}
    \big\|e^{it(-\Delta)^{a/2}}f\big\|_{L^2(w(x,t))}^2
    &=\int\big\|e^{it(-\Delta)^{a/2}}f\big\|_{L^2(w(\cdot,t))}^2dt\\
    &\leq C\int\bigg\|\bigg(\sum_k\big|P_ke^{it(-\Delta)^{a/2}}f\big|^2\bigg)^{1/2}\bigg\|_{L^2(w(\cdot,t))}^2dt\\
    &=C\sum_k\big\|e^{it(-\Delta)^{a/2}}P_kf\big\|_{L^2(w(x,t))}^2.
    \end{align*}
Here the constant $C$ which follows from the Littlewood-Paley theorem is generally depending on the weight $w$ by
$C=C_w=C_{A_2}$, but in our case $C_{A_2}$ is independent of $w$ (see the first paragraph below Proposition \ref{prop_par}).
On the other hand, since $P_kP_jf=0$ if $|j-k|\geq2$, it follows from \eqref{freq_par} that
    \begin{align*}
    \sum_k\big\|e^{it(-\Delta)^{a/2}}P_kf\big\|_{L^2(w(x,t))}^2
    &=\sum_k\big\|e^{it(-\Delta)^{a/2}}P_k\big(\sum_{|j-k|\leq1}P_jf\big)\big\|_{L^2(w(x,t))}^2\\
    &\leq C\|w\|_{\mathfrak{L}^{\beta,p}_{a-par}}\sum_k2^{k(\beta -a)}\big\|\sum_{|j-k|\leq1}P_jf\big\|_2^2
    \end{align*}
if $\beta>1+a/p$, $a>1$ and $1<p\leq(n+a)/\beta$.
Consequently, by taking $\beta=a+2s$, we get
    $$
    \big\|e^{it(-\Delta)^{a/2}} f \big\|_{L^2(w(x,t))}
    \leq C\|w\|_{\mathfrak{L}^{a+2s,p}_{a-par}}^{1/2}\|f\|_{\dot{H}^s}
    $$
if\, $a>1$, $s>\frac a2(\frac1p-1)$ and $\max\{\frac{a}{a-1+2s},1\}<p\leq\frac{n+a}{a+2s}$, as desired.

The inhomogeneous estimate \eqref{inho_par} follows also from the same argument.
Indeed, by the Littlewood-Paley theorem as before,
one can see that
    \begin{align*}
    &\bigg\|\int_{0}^{t} e^{i(t-s)(-\Delta)^{a/2}}F(\cdot,s)ds\bigg\|_{L^2(w(x,t))}^2\\
    &\qquad\leq C\sum_{k}\bigg\|\int_{0}^{t}e^{i(t-s)(-\Delta)^{a/2}}P_k\big(\sum_{|j-k|\leq1}P_j F(\cdot,s)\big)ds \bigg\|_{L^2(w(x,t))}^2.
    \end{align*}
By using \eqref{base_par}, the right-hand side in the above is bounded by
    $$
    C\|w\|_{\mathfrak{L}^{\beta,p}_{a-par}}^2
    \sum_{k} 2^{k(\beta -a)} \big\|\sum_{|j-k|\leq1}P_j F\big\|_{L^2(w(x,t)^{-1})}^2
    $$
if $\beta>1+a/p$, $a>1$ and $1<p\leq(n+a)/\beta$.
Since $w(\cdot,t)^{-1}\in A_2(\mathbb{R}^n)$ if and only if $w(\cdot,t)\in A_2(\mathbb{R}^n)$,
by applying the Littlewood-Paley theorem again and taking $\beta=a$, this is now bounded by
$C\|w\|_{\mathfrak{L}^{a,p}_{a-par}}^2\|F\|_{L^2(w(x,t)^{-1})}^2$
if\, $a>1$ and $a/(a-1)<p\leq(n+a)/a$.
Consequently, we get \eqref{inho_par}.
Theorem \ref{thm_par} is now proved.

\section{Proof of Proposition \ref{prop_par}}\label{sec4}

In this section we prove Proposition \ref{prop_par}.
We first show \eqref{freq_par} assuming Lemma \ref{local_prop} which is proved in
Subsection \ref{sec4.2}, and then \eqref{base_par} follows from a similar argument
in Subsection \ref{sec4.3}.

\subsection{Proof of \eqref{freq_par}}\label{sec4.1}

From the scaling $(x,t)\rightarrow(\lambda x,\lambda^{a}t)$,
it is enough to show the following case where $k=0$:
\begin{equation}\label{freq2}
\big\|e^{it(-\Delta)^{a/2}} P_0 f \big\|_{L^2(w(x,t))}
\leq C\|w\|_{\mathfrak{L}^{\beta,p}_{a-par}}^{1/2}\|f\|_{L^2},
\end{equation}
where $\beta>1+a/p$, $a>1$ and $1<p\leq(n+a)/\beta$.
In fact, once we show this estimate, we get
    \begin{align*}
    \big\|e^{it(-\Delta)^{a/2}}P_k f\big\|_{L^2(w(x,t))}^2
    &\leq C2^{-kn}2^{-ak}\big\|e^{it(-\Delta)^{a/2}}P_0(f(2^{-k}\cdot))\big\|_{L^2(w(2^{-k}x,2^{-ak}t))}^2\\
    &\leq C2^{-kn}2^{-ak}\|w(2^{-k}x,2^{-ak}t)\|_{\mathfrak{L}^{\beta,p}_{a-par}}\|f(2^{-k}\cdot)\|_2^2\\
    &\leq C2^{k(\beta -a)}\|w\|_{\mathfrak{L}^{\beta,p}_{a-par}}\|f\|_2^2
    \end{align*}
as desired.

Now, by duality, \eqref{freq2} is equivalent to
    \begin{equation*}
    \bigg\|\int_{\mathbb{R}} e^{-is(-\Delta)^{a/2}}P_0F(\cdot,s)ds\bigg\|_{L_x^2}
    \leq C\|w\|_{\mathfrak{L}^{\beta,p}_{a-par}}^{1/2}\|F\|_{L^2(w^{-1})},
    \end{equation*}
where we only use functions $F$ which are radial with respect to $x$-variable,
since $f$ is radial and the Schr\"odinger group evaluation and the Fourier projections keep the radial property.
Then, by using the usual $TT^*$ argument it is enough to show the following bilinear form estimate
    \begin{align}\label{bfs00}
    \nonumber\bigg|\bigg\langle\int_{\mathbb{R}} \Big( e^{i(t-s)(-\Delta)^{a/2}}P_0^2F(\cdot,s) \Big)(x) &ds,G(x,t)\bigg\rangle\bigg|\\
    &\leq C\|w\|_{\mathfrak{L}^{\beta,p}_{a-par}}\|F\|_{L^2(w^{-1})}\|G\|_{L^2(w^{-1})}
    \end{align}
for $\beta>1+a/p$, $a>1$ and $1<p\leq(n+a)/\beta$.
Of course, $F$ and $G$ are assumed here to be radial with respect to the space variable $x$.
For this estimate, we first decompose the involved functions into spatial-localized pieces as follows:
    $$
    F(x,t)= \chi_{[0,1)}(|x|) F(x,t) + \sum_{j=1}^{\infty} \chi_{[2^{j-1},2^j)}(|x|) F(x,t)
    $$
and
    $$
    G(x,t)= \chi_{[0,1)}(|x|) G(x,t) + \sum_{k=1}^{\infty} \chi_{[2^{k-1},2^k)}(|x|) G(x,t).
    $$
For simplicity, we set
\begin{equation}\label{qet}
F_0 =\chi_{[0,1)}(|x|)F,\quad F_j =\chi_{[2^{j-1},2^j)}(|x|)F,\quad j\geq1,
\end{equation}
and
\begin{equation}\label{qet2}
G_0 =\chi_{[0,1)}(|x|)G,\quad G_k =\chi_{[2^{k-1},2^k)}(|x|)G,\quad k\geq1.
\end{equation}
Then, by using this decomposition we are reduced to showing that
\begin{align}\label{bilin_par}
\nonumber\sum_{j,k=0}^{\infty} \bigg|\bigg\langle\int_{\mathbb{R}} \Big( e^{i(t-s)(-\Delta)^{a/2}}P_0^2F_j(\cdot,s) &\Big)(x) ds, G_k(x,t)\bigg\rangle\bigg|\\
&\leq C\|w\|_{\mathfrak{L}^{\beta,p}_{a-par}}\|F\|_{L^2(w^{-1})}\|G\|_{L^2(w^{-1})}
\end{align}
for $\beta>1+a/p$, $a>1$ and $1<p\leq(n+a)/\beta$.

To show \eqref{bilin_par}, we assume for the moment the following three estimates for $a>1$ which will be shown later:
\begin{itemize}
\item For $j,k\geq0$,
    \begin{equation}\label{inter_1}
    \bigg|\bigg\langle\int_{\mathbb{R}} \Big( e^{i(t-s)(-\Delta)^{a/2}}P_0^2 F_j(\cdot,s)\Big)(x)ds, G_k(x,t) \bigg\rangle\bigg|
    \lesssim 2^{\frac{1}{2}(j+k)}\|F\|_{L^2}\|G\|_{L^2}.
    \end{equation}
\item For $|j-k|\leq1$,
\begin{align}\label{inter_3'}
\nonumber\bigg|\bigg\langle\int_{\mathbb{R}} \Big( e^{i(t-s)(-\Delta)^{a/2}}&P_0^2F_j(\cdot,s)\Big)(x)ds, G_k(x,t)\bigg\rangle\bigg|\\
\lesssim&2^{(\frac{a+1}{2}-\frac{\beta p}{2})(j+k)}
\|w\|_{\mathfrak{L}^{\beta,p}_{a-par}}^p\|F\|_{L^2(w^{-p})}\|G\|_{L^2(w^{-p})}.
\end{align}
\item For $|j-k|>1$,
\begin{align}\label{inter_3}
\nonumber\bigg|\bigg\langle\int_{\mathbb{R}} \Big(e^{i(t-s)(-\Delta)^{a/2}}&P_0^2F_j(\cdot,s)\Big)(x)ds, G_k(x,t) \bigg\rangle\bigg|\\
\nonumber&\quad\lesssim2^{(\frac{2a+1}{4}-\frac{\beta p}{2})(j+k)}2^{-\frac{1}{4}|j-k|}
2^{\max(0,\frac12[M(j,k)-am(j,k)])}\\
&\qquad\qquad\times\|w\|_{\mathfrak{L}^{\beta,p}_{a-par}}^p\|F\|_{L^2(w^{-p})}\|G\|_{L^2(w^{-p})},
\end{align}
where $M(j,k):=\max(j,k)$ and $m(j,k):=\min(j,k)$.
\end{itemize}
When $|j-k|\leq1$, by the bilinear interpolation (see Lemma \ref{inter}) between \eqref{inter_1} and \eqref{inter_3'}, it follows that
\begin{align}\label{sdf}
\nonumber\bigg|\bigg\langle\int_{\mathbb{R}} \Big(e^{i(t-s)(-\Delta)^{a/2}}
&P_0^2F_j(\cdot,s)\Big)(x)ds, G_k(x,t)\bigg\rangle\bigg| \\
&\lesssim2^{(\frac{1}{2} +\frac{a}{2p} -\frac{\beta}{2})(j+k)}
\|w\|_{\mathfrak{L}^{\beta,p}_{a-par}} \|F\|_{L^2(w^{-1})} \|G\|_{L^2(w^{-1})}
\end{align}
for $a>1$ and $1<p\leq(n+a)/\beta$.
Indeed, let $T$ be a bilinear vector-valued operator defined by
$$T(F,G)=\bigg\{\bigg\langle\int_{\mathbb{R}} \Big(e^{i(t-s)(-\Delta)^{a/2}}
P_0^2F_j(\cdot,s)\Big)(x)ds, G_k(x,t)\bigg\rangle\bigg\}_{j\geq0}$$
for fixed $k\geq0$.
Then, \eqref{inter_1} and \eqref{inter_3'} are equivalent to
$$T:L^2\times L^2\rightarrow \ell_\infty^{\gamma_0}
\quad\text{and}\quad
T:L^2(w^{-p})\times L^2(w^{-p})\rightarrow \ell_\infty^{\gamma_1}$$
with the operator norms $2^{k/2}$ and
$2^{(\frac{a+1}{2}-\frac{\beta p}{2})k}\|w\|_{\mathfrak{L}^{\beta,p}_{a-par}}^p$,
respectively, where
$\gamma_0=2^{-j/2}$ and $\gamma_1=2^{-(\frac{a+1}{2}-\frac{\beta p}{2})j}$.
Now, by applying Lemma \ref{inter} with $\theta=1/p$, $q=\infty$ and $p_1=p_2=2$, we get
$$T:(L^2,L^2(w^{-p}))_{1/p,2}\times(L^2,L^2(w^{-p}))_{1/p,2}
\rightarrow(\ell_\infty^{\gamma_0},\ell_\infty^{\gamma_1})_{1/p,\infty}$$
for $1<p<\infty$, with the operator norm
$$2^{\frac k2(1-\frac1p)}2^{\frac1p(\frac{a+1}{2}-\frac{\beta p}{2})k}\|w\|_{\mathfrak{L}^{\beta,p}_{a-par}}
=2^{(\frac{1}{2} +\frac{a}{2p} -\frac{\beta}{2})k}\|w\|_{\mathfrak{L}^{\beta,p}_{a-par}}.$$
Finally, using the real interpolation space identities in Lemma \ref{id},
this implies that
$$T:L^2(w^{-1})\times L^2(w^{-1})\rightarrow\ell_\infty^{\gamma}$$
with the operator norm $2^{(\frac{1}{2} +\frac{a}{2p}-\frac{\beta}{2})k}\|w\|_{\mathfrak{L}^{\beta,p}_{a-par}}$
and
$\gamma=(1-\frac1p)\gamma_0+\frac1p\gamma_1=2^{-(\frac{1}{2} +\frac{a}{2p}-\frac{\beta}{2})j}$.
Clearly, this is equivalent to \eqref{sdf}.
Now, if $\beta>1+a/p$, we get from \eqref{sdf} that
\begin{align}\label{ert}
\nonumber\sum_{\{|j-k|\leq1\}} \bigg|\bigg\langle\int_{\mathbb{R}} \Big(e^{i(t-s)(-\Delta)^{a/2}}&P_0^2F_j(\cdot,s)\Big)(x)ds, G_k(x,t)\bigg\rangle\bigg|\\
&\leq C\|w\|_{\mathfrak{L}^{\beta,p}_{a-par}}\|F\|_{L^2(w^{-1})}\|G\|_{L^2(w^{-1})}.
\end{align}

On the other hand, when $|j-k|>1$,
by the bilinear interpolation between \eqref{inter_1} and \eqref{inter_3} as above,
it follows that
\begin{align}\label{dfg}
\nonumber\bigg|\bigg\langle\int_{\mathbb{R}} \Big(e^{i(t-s)(-\Delta)^{a/2}}
&P_0^2 F_j(\cdot,s)\Big)(x)ds, G_k(x,t)\bigg\rangle\bigg|\\
\nonumber&\lesssim2^{(\frac{1}{2}+\frac{2a-1}{4p}-\frac{\beta}{2})(j+k)}2^{-\frac{1}{4p}|j-k|}
2^{\frac1p\max(0,\frac12[M(j,k)-am(j,k)])}\\
&\qquad\qquad\times\|w\|_{\mathfrak{L}^{\beta,p}_{a-par}}\|F\|_{L^2(w^{-1})}\|G\|_{L^2(w^{-1})}
\end{align}
for $a>1$ and $1<p\leq(n+a)/\beta$.
Now we divide cases into $j\geq k$ and $j\leq k$.
Then, when $|j-k|>1$ and $j\geq k$, from \eqref{dfg} we see that
\begin{align*}
\sum_{\{|j-k|>1,j\geq k\}}\bigg|\bigg\langle\int_{\mathbb{R}}& \Big(e^{i(t-s)(-\Delta)^{a/2}}
P_0^2 F_j(\cdot,s)\Big)(x)ds, G_k(x,t)\bigg\rangle\bigg|\\
&\lesssim \sum_{\{|j-k|>1,j\geq k\}}2^{(\frac{1}{2} +\frac{2a-1}{4p} -\frac{\beta}{2})(j+k)} 2^{-\frac{1}{4p}(j-k)}2^{\max(0,\frac{1}{2p}(j- ak))}\\
&\qquad\qquad\qquad\qquad\times\|w\|_{\mathfrak{L}^{\beta,p}_{a-par}}\|F\|_{L^2(w^{-1})}\|G\|_{L^2(w^{-1})}.
\end{align*}
Since the right-hand side in the above is decomposed as
    \begin{align*}
    \sum_{k=0}^{\infty}&2^{(\frac{1}{2}+\frac{a}{2p}-\frac{\beta}{2})k}
        \sum_{j=k+2}^{ak}2^{(\frac{1}{2}+\frac{a-1}{2p}-\frac{\beta}{2})j}
        \|w\|_{\mathfrak{L}^{\beta,p}_{a-par}}\|F\|_{L^2(w^{-1})}\|G\|_{L^2(w^{-1})}\\
    &+\sum_{k=0}^{\infty}2^{(\frac{1}{2}-\frac{\beta}{2})k}
        \sum_{j=ak}^{\infty}2^{(\frac{1}{2}+\frac{a}{2p}-\frac{\beta}{2})j}
        \|w\|_{\mathfrak{L}^{\beta,p}_{a-par}}\|F\|_{L^2(w^{-1})}\|G\|_{L^2(w^{-1})},
    \end{align*}
we get
\begin{align}\label{ert2}
\nonumber\sum_{\{|j-k|>1,j\geq k\}} \bigg|\bigg\langle\int_{\mathbb{R}} \Big(e^{i(t-s)(-\Delta)^{a/2}}&P_0^2F_j(\cdot,s)\Big)(x)ds, G_k(x,t)\bigg\rangle\bigg|\\
&\leq C\|w\|_{\mathfrak{L}^{\beta,p}_{a-par}}\|F\|_{L^2(w^{-1})}\|G\|_{L^2(w^{-1})}
\end{align}
if $\beta>1+a/p$.
Obviously, when $|j-k|>1$ and $j\leq k$, we get similarly
\begin{align}\label{ert3}
\nonumber\sum_{\{|j-k|>1,j\leq k\}} \bigg|\bigg\langle\int_{\mathbb{R}} \Big(e^{i(t-s)(-\Delta)^{a/2}}&P_0^2F_j(\cdot,s)\Big)(x)ds, G_k(x,t)\bigg\rangle\bigg|\\
&\leq C\|w\|_{\mathfrak{L}^{\beta,p}_{a-par}}\|F\|_{L^2(w^{-1})}\|G\|_{L^2(w^{-1})}
\end{align}
for $\beta>1+a/p$, $a>1$ and $1<p\leq(n+a)/\beta$.
Combining \eqref{ert}, \eqref{ert2} and \eqref{ert3}, we now obtain the desired estimate \eqref{bilin_par}.

\subsubsection{Proofs of \eqref{inter_1}, \eqref{inter_3'} and \eqref{inter_3}.}\label{sec4.1.1}
It remains to show the three estimates \eqref{inter_1}, \eqref{inter_3'} and \eqref{inter_3}.
These estimates are derived from the following lemma which will be shown in Subsection \ref{sec4.2}.

\begin{lem}\label{local_prop}
Let $n\geq2$.
For integers $j,k\geq0$, let $F_j$ and $G_k$ be given as in \eqref{qet} and \eqref{qet2}, respectively,
which are radial functions on $\mathbb{R}^{n+1}$ with respect to the spatial variable $x$.
If\, $a>1$, we then have the following three estimates:
\begin{itemize}
\item For $j,k\geq0$,
    \begin{equation}\label{interpolation_value_2}
    \bigg| \bigg\langle \int_{\mathbb{R}} \Big(e^{i(t-s)(-\Delta)^{a/2}}P_0^2 F_j(\cdot,s)\Big)(x)ds, G_k(x,t) \bigg\rangle \bigg|
    \lesssim 2^{\frac{1}{2}(j+k)}
        \|F_j\|_{L^{2}_{x,t}} \|G_k\|_{L^{2}_{x,t}}.
    \end{equation}
\item For $|j-k|\leq1$,
    \begin{align}\label{interpolation_value_3}
    \nonumber\noindent\bigg| \bigg\langle \int_{\mathbb{R}} \Big(e^{i(t-s)(-\Delta)^{a/2}}P_0^2 F_j(\cdot,s)\Big)(x)ds, &G_k(x,t) \bigg\rangle \bigg|\\
    &\lesssim 2^{-\frac{n-1}{2}(j+k)} \|F_j\|_{L^{1}_{x,t}} \|G_k\|_{L^{1}_{x,t}} .
    \end{align}
\item For $|j-k|>1$,
    \begin{align}\label{interpolation_value}
    \nonumber\bigg| \bigg\langle \int_{\mathbb{R}} \Big(e^{i(t-s)(-\Delta)^{a/2}}P_0^2 F_j(\cdot,s)&\Big)(x)ds,G_k(x,t) \bigg\rangle \bigg|\\
    &\lesssim 2^{-\frac{2n-1}{4}(j+k)} 2^{-\frac{1}{4}|j-k|}
        \|F_j\|_{L^{1}_{x,t}} \|G_k\|_{L^{1}_{x,t}}.
    \end{align}
    \end{itemize}
\end{lem}

Indeed, the estimate \eqref{inter_1} is just the same as \eqref{interpolation_value_2}.
From now on, we deduce \eqref{inter_3'} and \eqref{inter_3} from \eqref{interpolation_value_3}
and \eqref{interpolation_value}, respectively.
For fixed $j,k\geq0$, we denote $R= \max(2^j, 2^k)$,
and we set
\begin{equation*}
\phi_{\nu}^0(t)=\chi_{[\nu-R,\nu+R)}(t)
\end{equation*}
and for $l\geq1$
\begin{equation*}
\phi_{\nu+}^l(t)=\chi_{[\nu+2^{l-1}R,\nu+2^lR)}(t),\quad
\phi_{\nu-}^l(t)=\chi_{[\nu-2^{l}R,\nu-2^{l-1}R)}(t).
\end{equation*}
Then we may write
\begin{align}\label{ell}
\bigg\langle&\int_{\mathbb{R}} \Big(e^{i(t-s)(-\Delta)^{a/2}}P_0^2F_j(\cdot,s)\Big)(x)ds, G_k(x,t)\bigg\rangle\\
\nonumber&=\sum_{\nu\in R\mathbb{Z}}
\bigg\langle\int_{\mathbb{R}} \Big(e^{i(t-s)(-\Delta)^{a/2}}P_0^2F_j(\cdot,s)\Big)(x)ds,\phi_{\nu}^0 G_k(x,t)\bigg\rangle\\
\nonumber&=\sum_{\nu\in R\mathbb{Z}}
\bigg\langle \int_{\mathbb{R}} \Big(e^{i(t-s)(-\Delta)^{a/2}}
P_0^2 \big([\phi_{\nu}^0+\sum_{l=1}^{\infty}(\phi_{\nu+}^l+\phi_{\nu-}^l)]F_j \big)(\cdot,s)\Big)(x)ds,\phi_\nu^0 G_k(x,t)\bigg\rangle.
\end{align}
To show \eqref{inter_3'} and \eqref{inter_3}, we assume for the moment that
\begin{align}\label{Lemma_case:l>1}
\nonumber\sum_{\nu\in R\mathbb{Z}}\sum_{l=2}^{\infty}
\bigg|\bigg\langle\int_{\mathbb{R}} \Big(e^{i(t-s)(-\Delta)^{a/2}}
&P_0^2(\phi_{\nu}^l F_j)(\cdot,s)\Big)(x)ds, \phi_\nu^0 G_k(x,t)\bigg\rangle\bigg|\\
\leq &C_M 2^{-(j+k)M}
\|w\|_{\mathfrak{L}_{a-par}^{\beta,p}}^p\|F\|_{L^2(w^{-p})}\|G\|_{L^2(w^{-p})}
\end{align}
for a sufficiently large number $M>0$,
where $\phi_{\nu}^l$ stands for $\phi_{\nu+}^l$ or $\phi_{\nu-}^l$.
This will be shown in the end of this subsection.
Then by \eqref{Lemma_case:l>1}, we only need to bound
\begin{equation}\label{bound}
\sum_{\nu\in R\mathbb{Z}}
\bigg\langle \int_{\mathbb{R}} \Big(e^{i(t-s)(-\Delta)^{a/2}}
P_0^2(\phi_{\nu}F_j)(\cdot,s)\Big)(x)ds, \phi_\nu^0 G_k(x,t)\bigg\rangle,
\end{equation}
where $\phi_{\nu}(t):=(\phi_{\nu}^0+\phi_{\nu+}^1+\phi_{\nu-}^1)(t)=\chi_{[v-2R, v+2R)}(t)$.

To show the bound \eqref{inter_3'} for \eqref{bound},
from \eqref{interpolation_value_3} we first see that
$$
\begin{aligned}
\bigg| \bigg\langle \int_{\mathbb{R}} \Big(e^{i(t-s)(-\Delta)^{a/2}}P_0^2 (\phi_\nu F_j)(\cdot,s)&\Big)(x)ds, \phi_\nu^0 G_k(x,t) \bigg\rangle \bigg| \\
&\lesssim 2^{-\frac{n-1}{2}(j+k)}\big\| \phi_\nu F_j \big\|_{L^1} \big\| \phi_\nu^0 G_k \big\|_{L^1},
\end{aligned}
$$
and note that
\begin{equation}\label{case:l=0_Holder2}
\big\|\phi_\nu F_j\big\|_{L^1}
\leq\big\|\phi_\nu F\big\|_{L^2(w^{-p})}
\big\|\phi_\nu(t)\chi_{[2^j,2^{j+1})}(|x|)w(x,t)^{p/2}\big\|_{L^2}
\end{equation}
and
\begin{equation}\label{case:l=0_Holder23}
\big\|\phi_\nu^0 G_k\big\|_{L^1}
\leq\big\|\phi_\nu^0 G\big\|_{L^2(w^{-p})}
\big\|\phi_\nu^0(t)\chi_{[2^k,2^{k+1})}(|x|)w(x,t)^{p/2}\big\|_{L^2}.
\end{equation}
Then we get
\begin{align}\label{eop}
\nonumber\bigg| \bigg\langle \int_{\mathbb{R}} \Big(e^{i(t-s)(-\Delta)^{a/2}}&P_0^2 (\phi_\nu F_j)(\cdot,s)\Big)(x)ds, \phi_\nu^0 G_k(x,t) \bigg\rangle \bigg|\\
\nonumber\lesssim &2^{-\frac{n-1}{2}(j+k)}
\big\| \phi_\nu(t) \chi_{[2^j,2^{j+1})}(|x|) w^{p/2} \big\|_{L^2}
\big\| \phi_\nu F  \big\|_{L^2(w^{-p})}\\
&\qquad\times\big\| \phi_\nu^0(t) \chi_{[2^k,2^{k+1})}(|x|) w^{p/2} \big\|_{L^2}
\big\| \phi_\nu^0 G  \big\|_{L^2(w^{-p})}.
\end{align}
Since we are assuming $|j-k|\leq1$, $R=\max(2^j,2^k)=C2^j$.
Hence we see that
    \begin{align}\label{case:l=0_j=k_2}
    \nonumber\big\| \phi_\nu(t)\chi_{[2^j,2^{j+1})}(|x|)w(x,t)^{p/2}\big\|_{L^2}^2
    &=\int_{\nu-2C2^j}^{\nu+2C2^j}\int_{|x|\in [2^j,2^{j+1})} w(x,t)^p dxdt\\
    \nonumber&\leq \int_{\nu-2C2^{aj}}^{\nu+2C2^{aj}} \int_{|x|\in [2^j,2^{j+1})} w(x,t)^p dxdt \\
    &\leq C2^{j(n+a -\beta p)} \| w \|_{\mathfrak{L}^{\beta,p}_{a-par}}^p
    \end{align}
from the definition of the $a$-parabolic Morrey-Campanato class.
Similarly,
\begin{equation}\label{case:l=0_j=k_3}
    \big\| \phi_\nu^0(t) \chi_{[2^k,2^{k+1})}(|x|) w(x,t)^{p/2} \big\|_{L^2}^2
    \leq 2^{k(n+a -\beta p)} \| w \|_{\mathfrak{L}^{\beta,p}_{a-par}}^p.
    \end{equation}
By combining \eqref{eop}, \eqref{case:l=0_j=k_2} and \eqref{case:l=0_j=k_3},
it follows now that
\begin{align*}
\bigg| \bigg\langle \int_{\mathbb{R}}& \Big(e^{i(t-s)(-\Delta)^{a/2}}P_0^2 (\phi_\nu F_j)(\cdot,s)\Big)(x)ds, \phi_\nu^0 G_k(x,t) \bigg\rangle \bigg|\\
&\lesssim 2^{-\frac{n-1}{2}(j+k)}
2^{\frac12(j+k)(n+a -\beta p)}\| w \|_{\mathfrak{L}^{\beta,p}_{a-par}}^p
\big\| \phi_\nu F  \big\|_{L^2(w^{-p})}\big\| \phi_\nu^0 G  \big\|_{L^2(w^{-p})}.
\end{align*}
Consequently, we get the desired bound
\begin{align}\label{eop3}
\nonumber\sum_{\nu \in R \mathbb{Z}}\bigg| \bigg\langle \int_{\mathbb{R}} \Big(e^{i(t-s)(-\Delta)^{a/2}}&P_0^2 (\phi_\nu F_j)(\cdot,s)\Big)(x)ds, \phi_\nu^0 G_k(x,t) \bigg\rangle \bigg|\\
\lesssim &2^{(a+1-\beta p)(j+k)/2}
\|w\|_{\mathfrak{L}^{\beta,p}_{a-par}}^p\|F\|_{L^2(w^{-p})}\|G\|_{L^2(w^{-p})}
\end{align}
using the Cauchy-Schwarz inequality in $\nu$ with the trivial estimates
\begin{equation}\label{131}
\sum_{\nu \in R \mathbb{Z}}\big\|\phi_\nu F \big\|_{L^2(w^{-p})}^2
\leq C\|F\|_{L^2(w^{-p})}^2
\end{equation}
and
\begin{equation}\label{1312}
\sum_{\nu \in R \mathbb{Z}}\big\|\phi_\nu^0 G \big\|_{L^2(w^{-p})}^2
\leq C\|G\|_{L^2(w^{-p})}^2.
\end{equation}
By \eqref{Lemma_case:l>1} and \eqref{eop3}, we obtain \eqref{inter_3'}.

Now we have to show the bound \eqref{inter_3} for \eqref{bound}.
For simplicity, we will consider the case $j\geq k$ only,
because the other case $j\leq k$ can be shown clearly in the same way.
From \eqref{interpolation_value}, we first see that
    \begin{align*}
    \bigg| \bigg\langle \int_{\mathbb{R}} \Big(e^{i(t-s)(-\Delta)^{a/2}}P_0^2 (\phi_\nu F_j)&(\cdot,s)\Big)(x)ds, \phi_\nu^0 G_k(x,t) \bigg\rangle \bigg|\\
    &\lesssim 2^{(j+k)(-\frac{2n-1}{4})} 2^{-\frac{1}{4}|j-k|}
    \big\| \phi_\nu F_j \big\|_{L^1} \big\| \phi_\nu^0 G_k \big\|_{L^1}.
    \end{align*}
Then by \eqref{case:l=0_Holder2} and \eqref{case:l=0_Holder23}, it follows that
  \begin{align}\label{sfk}
\nonumber\bigg| \bigg\langle \int_{\mathbb{R}}& \Big(e^{i(t-s)(-\Delta)^{a/2}}P_0^2 (\phi_\nu F_j)(\cdot,s)\Big)(x)ds, \phi_\nu^0 G_k(x,t) \bigg\rangle \bigg|\\
\nonumber&\lesssim 2^{(j+k)(-\frac{2n-1}{4})} 2^{-\frac{1}{4}|j-k|}
\big\| \phi_\nu(t) \chi_{[2^j,2^{j+1})}(|x|) w^{p/2} \big\|_{L^2}
\big\| \phi_\nu F  \big\|_{L^2(w^{-p})}\\
&\qquad\times\big\| \phi_\nu^0(t) \chi_{[2^k,2^{k+1})}(|x|) w^{p/2} \big\|_{L^2}
\big\| \phi_\nu^0 G  \big\|_{L^2(w^{-p})}.
\end{align}
Since $R=\max(2^j,2^k)=2^j$, we see that
\begin{equation}\label{case:l=0_j=k_1}
\big\| \phi_\nu(t)\chi_{[2^j,2^{j+1})}(|x|) w(x,t)^{p/2}\big\|_{L^2}^2
\leq C2^{j(n+a -\beta p)}\|w\|_{\mathfrak{L}^{\beta,p}_{a-par}}^p
\end{equation}
as in \eqref{case:l=0_j=k_2}.
Now we claim that
\begin{equation}\label{qsc}
\big\|\phi_\nu^0(t)\chi_{[2^k,2^{k+1})}(|x|)w(x,t)^{p/2} \big\|_{L^2}^2\leq
C2^{\max(0,j-ak)}2^{k(n+a-\beta p)}\|w\|_{\mathfrak{L}^{\beta,p}_{a-par}}^p.
\end{equation}
Indeed, when $j-ak\geq0$,
    \begin{align*}
    \big\| \phi_\nu^0(t) \chi_{[2^k,2^{k+1})}(|x|) w(x,t)^{p/2} &\big\|_{L^2}^2
=\int_{\nu}^{\nu+2^j} \int_{|x|\in [2^k,2^{k+1})} w(x,t)^p dxdt\\
    &\leq \sum_{m=0}^{\textbf{[}2^{j-ak}-1\textbf{]}} \int_{\nu+m2^{ak}}^{\nu+(m+1)2^{ak}} \int_{|x|\in [2^k,2^{k+1})} w(x,t)^p dxdt \\
    &\leq C2^{j-ak} 2^{k(n+a -\beta p)} \| w \|_{\mathfrak{L}^{\beta,p}_{a-par}}^p,
    \end{align*}
where $\textbf{[}2^{j-ak}-1\textbf{]}$ denotes the least integer
greater than or equal to $2^{j-ak}-1$.
On the other hand, when $j-ak\leq0$,
    \begin{align*}
    \big\| \phi_\nu^0(t) \chi_{[2^k,2^{k+1})}(|x|) w(x,t)^{p/2} \big\|_{L^2}^2
    &= \int_{\nu}^{\nu+2^j} \int_{|x|\in [2^k,2^{k+1})} w(x,t)^p dxdt \\
    &\leq \int_{\nu}^{\nu+2^{ak}} \int_{|x|\in [2^k,2^{k+1})} w(x,t)^p dxdt \\
    &\leq C2^{k(n+a -\beta p)} \| w \|_{\mathfrak{L}^{\beta,p}_{a-par}}^p.
    \end{align*}
The claim \eqref{qsc} is proved.
By combining \eqref{sfk}, \eqref{case:l=0_j=k_1} and \eqref{qsc}, it follows now that
\begin{align*}
\bigg|\bigg\langle\int_{\mathbb{R}} \Big(e^{i(t-s)(-\Delta)^{a/2}}&P_0^2(\phi_\nu F_j)(\cdot,s)\Big)(x)ds, \phi_\nu^0 G_k(x,t)\bigg\rangle\bigg|\\
&\qquad\lesssim 2^{(\frac{2a+1}{4}-\frac{\beta p}{2})(j+k)}2^{-\frac{1}{4}(j-k)}
2^{\max(0,\frac12(j-ak))}\\
&\qquad\qquad\times\|w\|_{\mathfrak{L}^{\beta,p}_{a-par}}^p\big\|\phi_\nu F\big\|_{L^2(w^{-p})}\big\|
\phi_\nu^0G \big\|_{L^2(w^{-p})}.
\end{align*}
Using the Cauchy-Schwarz inequality in $\nu$ with \eqref{131} and \eqref{1312},
we get
\begin{align*}
\sum_{\nu \in R \mathbb{Z}}\bigg|\bigg\langle\int_{\mathbb{R}} \Big(e^{i(t-s)(-\Delta)^{a/2}}&P_0^2(\phi_\nu F_j)(\cdot,s)\Big)(x)ds, \phi_\nu^0 G_k(x,t)\bigg\rangle\bigg|\\
&\quad\lesssim 2^{(\frac{2a+1}{4}-\frac{\beta p}{2})(j+k)}2^{-\frac{1}{4}(j-k)}
2^{\max(0,\frac12(j-ak))}\\
&\quad\qquad\times\|w\|_{\mathfrak{L}^{\beta,p}_{a-par}}^p\|F\|_{L^2(w^{-p})}\|G\|_{L^2(w^{-p})}
\end{align*}
as desired.

\begin{proof}[Proof of \eqref{Lemma_case:l>1}]
It remains to show the estimate \eqref{Lemma_case:l>1}.
First we write
    \begin{align*}
\int_{\mathbb{R}}\Big(e^{i(t-s)(-\Delta)^{a/2}}&P_0^2 (\phi_\nu^l F_j)(\cdot,s)\Big)(x)ds \\
    &= \int_{\mathbb{R}}\int_{\mathbb{R}^{n}}
        \bigg( \int_{\mathbb{R}^n} e^{i(x-y)\cdot\xi +i(t-s)|\xi|^a} \psi^2(\xi) d\xi \bigg)
        \phi_\nu^l F_j(y,s) dy ds.
    \end{align*}
From the support of $\phi_\nu^0 G_k$ and $\phi_\nu^l F_j$,
we may assume that $|x|\sim 2^k$, $|y|\sim 2^j$ and $|t-s|\sim 2^l R$ since $l\geq2$.
Then by the integration by parts, we easily see that
    \begin{equation}\label{use}
    \bigg| \int_{\mathbb{R}^n} e^{i(x-y)\cdot\xi +i(t-s)|\xi|^a} \psi^2(\xi) d\xi\bigg|
    \leq C_N (2^l R)^{-N}
    \end{equation}
for a sufficiently large number $N>0$.
Using this, we now get
    \begin{align}\label{case:l>1}
\nonumber\bigg| \bigg\langle \int_{\mathbb{R}} \Big(e^{i(t-s)(-\Delta)^{a/2}}&P_0^2 (\phi_\nu^l F_j)(\cdot,s)\Big)(x)ds, \phi_\nu^0 G_k(x,t) \bigg\rangle \bigg|\\
\nonumber&\leq \bigg\| \int_{\mathbb{R}} e^{i(t-s)(-\Delta)^{a/2}}P_0^2 (\phi_\nu^l F_j)(\cdot,s)ds \bigg\|_{L^{\infty}} \big\| \phi_\nu^0 G_k \big\|_{L^1}\\
    &\leq C_N (2^l R)^{-N}
        \big\| \phi_\nu^l F_j \big\|_{L^1} \big\| \phi_\nu^0 G_k \big\|_{L^1} .
    \end{align}
Next, by H\"older's inequality we note that
    \begin{align*}
    \big\| \phi_\nu^l F_j \big\|_{L^1}
    &= \iint \phi_\nu^l(t)\chi_{[2^{j-1},2^j)}(|x|)|F(x,t)|w(x,t)^{-p/2} w(x,t)^{p/2} dxdt \\
    &\leq \big\| \phi_\nu^l F  \big\|_{L^2(w^{-p})}
        \bigg( \int w(x,t)^p \phi_\nu^l(t) \chi_{B(0,2^{j+1})}(x) dxdt \bigg)^{1/2},
    \end{align*}
where $B(0,2^{j+1})$ denotes the ball in $\mathbb{R}^n$ centered at the origin with radius $2^{j+1}$.
Also, by the definition of $\mathfrak{L}^{\beta,p}_{a-par}$,
    \begin{align*}
    \int w(x,t)^p \phi_\nu^l(t) \chi_{B(0,2^{j+1})}(x) dxdt
    &\leq \int_{|t-\nu| \leq (2^l R)^a} \int_{|x| \leq 2^l R} w(x,t)^p dxdt \\
    &\leq C(2^l R)^{n+a -\beta p} \| w \|_{\mathfrak{L}^{\beta,p}_{a-par}}^p.
    \end{align*}
Hence it follows that
\begin{equation}\label{55}
    \big\| \phi_\nu^l F_j \big\|_{L^1}
    \leq C(2^l R)^{(n+a-\beta p)/2}\| w \|_{\mathfrak{L}^{\beta,p}_{a-par}}^{p/2}\big\| \phi_\nu^l F  \big\|_{L^2(w^{-p})}.
\end{equation}
Similarly,
\begin{equation}\label{552}
    \big\| \phi_\nu^0 G_k \big\|_{L^1}
    \leq CR^{(n+a-\beta p)/2}\|w\|_{\mathfrak{L}^{\beta,p}_{a-par}}^{p/2}\big\|\phi_\nu^0G \big\|_{L^2(w^{-p})}.
\end{equation}
Combining \eqref{case:l>1}, \eqref{55} and \eqref{552}, we conclude that
 \begin{align*}
\bigg| \bigg\langle \int_{\mathbb{R}}&\Big(e^{i(t-s)(-\Delta)^{a/2}}P_0^2 (\phi_\nu^l F_j)(\cdot,s)\Big)(x)ds, \phi_\nu^0 G_k(x,t) \bigg\rangle \bigg|\\
    &\leq C_N (2^l R)^{-N}(2^{l/2}R)^{n+a-\beta p}
    \| w \|_{\mathfrak{L}^{\beta,p}_{a-par}}^p\big\| \phi_\nu^l F  \big\|_{L^2(w^{-p})}
\big\|\phi_\nu^0G \big\|_{L^2(w^{-p})}.
    \end{align*}
Using this and the Cauchy-Schwarz inequality as before,
we finally get
    \begin{align}\label{endd}
\nonumber&\sum_{\nu \in R \mathbb{Z}} \sum_{l=2}^{\infty}
        \bigg| \bigg\langle \int_{\mathbb{R}} \Big(e^{i(t-s)(-\Delta)^{a/2}}P_0^2 (\phi_\nu^l F_j)(\cdot,s)\Big)(x)ds, \phi_\nu^0 G_k(x,t) \bigg\rangle \bigg| \\
\nonumber&\quad\leq C_N \sum_{l=2}^{\infty}
        (2^{l/2}R)^{n+a -\beta p} (2^l R)^{-N}
        \| w \|_{\mathfrak{L}^{\beta,p}_{a-par}}^p
        \sum_{\nu \in R \mathbb{Z}}  \big\| \phi_\nu^l F \big\|_{L^2(w^{-p})} \big\| \phi_\nu^0 G \big\|_{L^2(w^{-p})} \\
&\quad\lesssim C_N R^{n+a-\beta p-N}
        \| w \|_{\mathfrak{L}^{\beta,p}_{a-par}}^p \| F \|_{L^2(w^{-p})} \| G \|_{L^2(w^{-p})}.
    \end{align}
Here, to apply the Cauchy-Schwarz inequality, we have used the following trivial estimates:
    \begin{equation*}
    \sum_{\nu \in R \mathbb{Z}} \big\| \phi_\nu^l F \big\|_{L^2(w^{-p})}^2
    \leq C2^l \| F \|_{L^2(w^{-p})}^2
    \end{equation*}
and
 \begin{equation*}
    \sum_{\nu \in R \mathbb{Z}} \big\| \phi_\nu^0 G \big\|_{L^2(w^{-p})}^2
    \leq C\| G \|_{L^2(w^{-p})}^2.
    \end{equation*}
Since $N$ is sufficiently large and $R=\max(2^j,2^k)\geq2^{(j+k)/2}$,
\eqref{endd} implies directly the estimate \eqref{Lemma_case:l>1}.
\end{proof}


\subsection{Proof of Lemma \ref{local_prop}}\label{sec4.2}

Here we prove Lemma \ref{local_prop}.
First, we show \eqref{interpolation_value_3} and \eqref{interpolation_value},
and then we show \eqref{interpolation_value_2}.

\subsubsection{Proofs of \eqref{interpolation_value_3} and \eqref{interpolation_value}}
Let us first consider
$x= rx'$, $y=\lambda y'$ and $\xi=\rho\xi'$ for $x', y', \xi' \in S^{n-1}$,
where $r = |x|$, $\lambda=|y|$ and $\rho = |\xi|$.
Recall the fact\footnote{Here, $\sigma$ is the measure induced by the Lebesgue measure on $S^{n-1}$ and $J_m$ denotes the Bessel function with order $m$.} (see \cite{St}, p. 347) that
    \begin{equation}\label{spheri}
    \int_{S^{n-1}} e^{-ir\rho x'\cdot\xi'} d\sigma(x')=c_n (r\rho)^{-\frac{n-2}{2}} J_{\frac{n-2}{2}} (r\rho).
    \end{equation}
Using this as in (27) of \cite{CKS} and setting $\widetilde{F_j}(\lambda,s):=F_j(\lambda y',s)$,\footnote{Note here that $F_j(\lambda y',s)$ is independent of $y'\in S^{n-1}$
since $F_j$ is a radial function in the $x$ variable.}
one can easily see that
    \begin{equation}\label{0907}
    \chi_{I_k}(|x|) \int_{\mathbb{R}} \Big(e^{i(t-s) (-\Delta)^{a/2}} P_0^2 F_j(\cdot,s)\Big)(x) ds =
    \iint K_{jk}(r,\lambda,t-s) \lambda^{n-1} \widetilde{F_j}(\lambda,s) d\lambda ds
   \end{equation}
with $K_{jk}(r,\lambda,t)$, $j,k\geq0$, which is given as
$$
K_{jk}(r,\lambda,t)
=\frac{\chi_{I_k}(r)}{r^{\frac{n-2}{2}}}\frac{\chi_{I_j}(\lambda)}{\lambda^{\frac{n-2}{2}}}
\int e^{it\rho^{a}} \overline{J_{\frac{n-2}{2}}(r\rho)} J_{\frac{n-2}{2}}(\lambda\rho)\rho\psi(\rho)^2d\rho
$$
where $I_0=(0,1)$, and for $j,k\geq1$, $I_k=[2^{k-1},2^k)$ and $I_j=[2^{j-1},2^j)$.
Since
\begin{align*}
\bigg|\bigg\langle \int_{\mathbb{R}} \Big(&e^{i(t-s)(-\Delta)^{a/2}} P_0^2 F_j(\cdot,s)\Big)(x)ds, G_k(x,t)\bigg\rangle\bigg|\\
&\qquad\qquad\leq\sup_{r,t} \bigg| \iint K_{jk}(r,\lambda,t-s) \lambda^{n-1} \widetilde{F_j}(\lambda,s) d\lambda ds \bigg|\| G_k \|_{L^1_{x,t}} \\
&\qquad\qquad\leq \sup_{r,\lambda,t} |K_{jk}(r,\lambda,t)| \big\| \lambda^{n-1} \widetilde{F_j}(\lambda,s) \big\|_{L^1_{\lambda,s}}\| G_k \|_{L^1_{x,t}} \\
&\qquad\qquad\leq C \| K_{jk} \|_{L^\infty_{r,\lambda,t}} \| F_j \|_{L^1_{x,t}}\| G_k \|_{L^1_{x,t}},
\end{align*}
we are now reduced to showing that
    \begin{equation}\label{reduce}
    \| K_{jk} \|_{L^\infty_{r,\lambda,t}}
    \lesssim
    \begin{cases}
    2^{-\frac{n-1}{2}(j+k)},\\
    2^{-\frac{2n-1}{4}(j+k)}2^{-\frac{1}{4}|j-k|}\quad\text{particularly if}\quad|j-k|>1.
    \end{cases}
    \end{equation}

First we show the first bound in \eqref{reduce}.
For $n\geq2$, we see that
    \begin{equation}\label{Bessel_Grafakos}
    |J_{\frac{n-2}{2}}(r)|
    \leq C \min\{ r^{\frac{n-2}{2}}, r^{-\frac{1}{2}} \}
    \end{equation}
using the following known estimates for Bessel functions $J_{\nu}(r)$ (see \cite{G}, pp. 429-431):
For $\textrm{Re}\,\nu>-1/2$
   $$
    |J_{\nu}(r)|
    \lesssim
    \begin{cases}
    C_\nu r^{\textrm{Re}\,\nu}\quad\text{if}\quad 0<r\leq1,\\
    C_\nu r^{-1/2}\quad\text{if}\quad r\geq1.
    \end{cases}
    $$
Hence it follows from \eqref{Bessel_Grafakos} that
    \begin{align*}
    |K_{jk}(r,\lambda,t)|
    &\leq \frac{\chi_{I_k}(r)}{ r^{\frac{n-2}{2}} } \frac{\chi_{I_j}(\lambda)}{ \lambda^{\frac{n-2}{2}} }
        \int |J_{\frac{n-2}{2}} (r\rho)| |J_{\frac{n-2}{2}} (\lambda\rho)| \rho\phi(\rho)^2 d\rho \\
    &\lesssim \chi_{I_k}(r) \chi_{I_j}(\lambda)
    \min\{ 1, r^{-\frac{n-1}{2}} \} \min\{ 1, \lambda^{-\frac{n-1}{2}} \}.
    \end{align*}
From the supports of $\chi_{I_k}$ and $\chi_{I_j}$, this implies now the desired bound.

Now we turn to \eqref{reduce} for the case $|j-k|>1$.
We divide cases into the case $j,k\geq1$ and the case where $j=0$ or $k=0$.

\medskip

\noindent\textbf{The case $j,k\geq1$ when $|j-k|>1$.}
In this case, we will decompose $K_{jk}$ into four parts
based on the following asymptotic expansion of Bessel functions (see Lemma 3.4 in \cite{CKS}).
We also refer the reader to \cite{W} for the theory of Bessel functions.

\begin{lem}\label{bessel}
For $r>1$ and $Re\,\nu>-1/2$,
\begin{equation}\label{dembes}
J_{\nu}(r)=\frac{\sqrt{2}}{\sqrt{\pi r}}\cos(r-\frac{\nu\pi}2-\frac{\pi}4)
-\frac{(\nu-\frac12)\Gamma(\nu+\frac32)}{(2\pi)^{\frac12}(r)^{\frac32}\Gamma(\nu+\frac12)}
\sin(r-\frac{\nu\pi}2-\frac{\pi}4)+E_{\nu}(r),
\end{equation}
where
\begin{equation}\label{Bessel_error_1}
|E_{\nu}(r)|\leq
C_\nu r^{-\frac52}
\end{equation}
and
\begin{equation}\label{Bessel_error_2}
|\frac{d}{dr}E_{\nu}(r)|\leq C_\nu
(r^{-\frac52}+r^{-\frac72}).
\end{equation}
\end{lem}

Indeed, using this lemma as in (46) of \cite{CKS},
we may write
$$J_{\frac{n-2}{2}}(\lambda\rho) \overline{J_{\frac{n-2}{2}}(r\rho)} = \sum_{l=1}^4J_l(r,\lambda,\rho),$$
where
 \begin{align*}
J_1(r,\lambda,\rho)&=
(c_{n}(\lambda\rho)^{-\frac{1}{2}}+c_{n}(\lambda\rho)^{-\frac{3}{2}})
e^{\pm i\lambda\rho}(c_{n}(r\rho)^{-\frac{1}{2}}+c_{n}(r\rho)^{-\frac{3}{2}} )e^{\pm ir\rho},\\ J_2(r,\lambda,\rho)&=
(c_{n}(\lambda\rho)^{-\frac{1}{2}}+c_{n}(\lambda\rho)^{-\frac{3}{2}})
e^{\pm i\lambda\rho}E_{\frac{n-2}{2}}(r\rho),\\
J_3(r,\lambda,\rho)&=
(c_{n}(r\rho)^{-\frac{1}{2}}+c_{n}(r\rho)^{-\frac{3}{2}})e^{\pm ir\rho}E_{\frac{n-2}{2}}(\lambda\rho),\\
J_4(r,\lambda,\rho)&=E_{\frac{n-2}{2}}(\lambda\rho)E_{\frac{n-2}{2}}(r\rho).
\end{align*}
From this, $K_{jk}$ is now decomposed as $K_{jk}=\sum_{l=1}^4 K_{jk,l}$, with
$$K_{jk,l}(r,\lambda,t)=
\frac{\chi_{I_k}(r)}{r^{\frac{n-2}{2}}}
\frac{\chi_{I_j}(\lambda)}{\lambda^{\frac{n-2}{2}}}
\int e^{it\rho^{a}}J_l(r,\lambda,\rho)\rho\psi(\rho)^2d\rho.$$
Then we only need to show
\begin{equation}\label{mainbes}
\|K_{jk,l}\|_{L^{\infty}_{r,\lambda,t} }
    \lesssim 2^{-j\frac{2n-1}{4}}2^{-k\frac{2n-1}{4}}2^{-\frac{1}{4}|j-k|}
\end{equation}
for $l=1,2,3,4$.

\begin{rem}
As shown below, although the cases $l=2,3,4$, which follow from the error term $E_{\nu}(r)$ in the asymptotic expansion \eqref{dembes} of the Bessel function, would give a better bound than \eqref{mainbes},
the bound for the first case $l=1$, which follows from the first term in the expansion, dominates those better bounds.
But, if we use the usual expansion of the Bessel function having the error term on and after the second term,
the error estimates like \eqref{Bessel_error_1} and \eqref{Bessel_error_2} are not enough so that the first case dominates the other cases.
This is the reason why we compute the second term and have the terms on and after the third term as the error term in Lemma \ref{bessel} .
\end{rem}

For $l=4$, it follows easily from \eqref{Bessel_error_1} that
    $$
    \|K_{jk,4}\|_{L^{\infty}_{r,\lambda,t} }
    \lesssim 2^{-j\frac{n+3}{2}}2^{-k\frac{n+3}{2}}
    \leq 2^{-j\frac{2n-1}{4}}2^{-k\frac{2n-1}{4}}2^{-\frac{1}{4}|j-k|}.
    $$

Next, for $l=1$, we may show the desired bound for
\begin{equation}\label{firstbes}
\widetilde{K}_{jk,1}(r,\lambda,t)
= \frac{\chi_{I_k}(r)}{ r^{\frac{n-1}{2}} } \frac{\chi_{I_j}(\lambda)}{\lambda^{\frac{n-1}{2}} }
\int e^{it\rho^a\pm i\lambda\rho \pm ir\rho}\psi(\rho)^2d\rho,
\end{equation}
since the factors $(\lambda\rho)^{-\frac32}$ and $(r\rho)^{-\frac32}$ in $J_1$
would give a better boundedness
than $(\lambda\rho)^{-\frac12}$ and $(r\rho)^{-\frac12}$, respectively.
Now, applying Lemma \ref{VanCor+IntPart} with $\Phi(\rho)=\psi^2(\rho)$
and $R=|\lambda\pm r| \sim 2^{\max(j,k)}$ since $|j-k|>1$ and $j,k\geq1$, we get
    \begin{align*}
    \big| \widetilde{K}_{jk,1}(r,\lambda,t) \big|
    &\lesssim \frac{\chi_{I_k}(r)}{ r^{\frac{n-1}{2}} } \frac{\chi_{I_j}(\lambda)}{\lambda^{\frac{n-1}{2}}}\,2^{-\frac{1}{2}\max(j,k)}\\
    &\sim 2^{-j\frac{2n-1}{4}}2^{-k\frac{2n-1}{4}}2^{-\frac{1}{4}|j-k|},
    \end{align*}
as desired.

It remains to bound $K_{jk,2}$ and $K_{jk,3}$.
We shall consider only for $K_{jk,2}$
because the same argument used for $K_{jk,2}$ works clearly for $K_{jk,3}$.
Since the factor $(\lambda\rho)^{-\frac32}$ in $J_2$ would give a better boundedness
than $(\lambda\rho)^{-\frac12}$, we only need to show the desired bound for
    $$
    \widetilde{K}_{jk,2}(r,\lambda,t)
    = \frac{\chi_{I_k}(r)}{ r^{\frac{n-2}{2}} } \frac{\chi_{I_j}(\lambda)}{ \lambda^{\frac{n-1}{2}} }
        \int e^{it\rho^a\pm i\lambda\rho}E_{\frac{n-2}{2}}(r\rho) \rho^{\frac12}\psi(\rho)^2d\rho.
    $$
Applying Lemma \ref{VanCor+IntPart} with $R= |\lambda| \sim 2^j$
and $\Phi(\rho)= E_{\frac{n-2}{2}}(r\rho) \rho^{\frac12}\psi^2(\rho)$,
we get
$$
\big|\widetilde{K}_{jk,2}(r,\lambda,t) \big|
\lesssim \frac{\chi_{I_k}(r)}{ r^{\frac{n-2}{2}} } \frac{\chi_{I_j}(\lambda)}{ \lambda^{\frac{n-1}{2}} }
\,2^{-\frac{j}{2}} \big( \| \Phi \|_{L^\infty} + \| \Phi' \|_{L^\infty} \big).
$$
Using \eqref{Bessel_error_1} and \eqref{Bessel_error_2}, we notice here
that $\| \Phi \|_{L^\infty} + \|\Phi'\|_{L^\infty}\lesssim r^{-\frac{3}{2}}$
which follows from
    \begin{equation*}
    | E_{\frac{n-2}{2}} (r\rho) | \lesssim r^{-\frac{5}{2}}
    \end{equation*}
    and
    $$
    \Big| \frac{d}{d\rho} E_{\frac{n-2}{2}} (r\rho) \Big|
    =\Big| \frac{d}{dv} E_{\frac{n-2}{2}} (v)\frac{dv}{d\rho} \Big|
     \lesssim((r\rho)^{-5/2}+(r\rho)^{-7/2})r
    \lesssim r^{-\frac{3}{2}}.
    $$
Thus, we get
    \begin{align*}
    \big\|\widetilde{K}_{jk,2}\big\|_{L^{\infty}_{r,\lambda,t}}
    &\lesssim 2^{-j\frac{n-1}{2}}2^{-k\frac{n-2}{2}}2^{-\frac{j}{2}}2^{-k\frac{3}{2}}\\
    &\leq 2^{-j\frac{2n-1}{4}}2^{-k\frac{2n-1}{4}}2^{-\frac{1}{4}|j-k|}.
    \end{align*}

\medskip

\noindent\textbf{The case where $j=0$ or $k=0$ when $|j-k|>1$.}
In this case, we will use the following known fact (see \cite{G}, p. 426):
For $0\leq r<1$ and $Re\,\nu>-1/2$,
    \begin{equation}\label{Besse}
    |J_{\nu} (r)| \leq C_{\nu} r^{\nu}
    \quad\text{and}\quad
    \big|\frac{d}{dr}J_{\nu} (r)\big| \leq C_{\nu} r^{\nu-1}.
    \end{equation}
We consider only the case where $k=0$ and $j\geq1$
since the other case where $j=0$ and $k\geq1$ follows clearly from the same argument.
Now we have to show that for $j\geq1$
\begin{equation}\label{inft}
\|K_{j0}(r,\lambda,t)\|_{L^\infty_{r,\lambda,t}} \lesssim 2^{-j\frac{n}{2}}
\end{equation}
where
$$K_{j0}(r,\lambda,t)
=\frac{\chi_{I_0}(r)}{r^{\frac{n-2}{2}}}\frac{\chi_{I_j}(\lambda)}{\lambda^{\frac{n-2}{2}}}
\int e^{it\rho^{a}} \overline{J_{\frac{n-2}{2}}(r\rho)} J_{\frac{n-2}{2}}(\lambda\rho)\rho\psi(\rho)^2d\rho.$$
Notice from Lemma \ref{bessel} that
    \begin{equation}\label{jj}
    J_{\frac{n-2}{2}}(\lambda \rho)=
    (c_{n}(\lambda\rho)^{-\frac{1}{2}}+c_{n}(\lambda\rho)^{-\frac{3}{2}})e^{\pm i\lambda\rho}
    +E_{\frac{n-2}{2}} (\lambda\rho).
    \end{equation}
By \eqref{Bessel_error_1} and \eqref{Besse},
the part of $K_{j0}$ coming from $E_{\frac{n-2}{2}}(\lambda\rho)$ in $\eqref{jj}$ is bounded as follows:
\begin{equation*}
\bigg|\frac{\chi_{I_0}(r)}{r^{\frac{n-2}{2}}}\frac{\chi_{I_j}(\lambda)}{\lambda^{\frac{n-2}{2}}}
\int e^{it\rho^{a}} \overline{J_{\frac{n-2}{2}}(r\rho)} E_{\frac{n-2}{2}}(\lambda\rho)\rho\psi(\rho)^2 d\rho\bigg|\lesssim2^{-j\frac{n+3}{2}}\leq2^{-j\frac{n}{2}}.
\end{equation*}
Now we may consider only the part of $K_{j0}$ coming from $(\lambda\rho)^{-\frac{1}{2}}$,
because the factor $(\lambda\rho)^{-\frac32}$ in $\eqref{jj}$ would give a better boundedness
than $(\lambda\rho)^{-\frac12}$.
Namely, we have to show the bound \eqref{inft} for
    $$
    \widetilde{K}_{j0}(r,\lambda,t)
    =\frac{\chi_{I_0}(r)}{r^{\frac{n-2}{2}}}\frac{\chi_{I_j}(\lambda)}{\lambda^{\frac{n-1}{2}}}
    \int e^{it\rho^{a}\pm i\lambda\rho}
    \overline{J_{\frac{n-2}{2}}(r\rho)} \rho^{\frac12}\psi(\rho)^2d\rho.
    $$
Applying Lemma \ref{VanCor+IntPart} with $R= |\lambda| \sim 2^j$ and
$\Phi(\rho)= J_{\frac{n-2}{2}}(r\rho) \rho^{\frac12}\psi^2(\rho)$, and by \eqref{Besse},
we then get
$$\big|\widetilde{K}_{j0}(r,\lambda,t) \big|
    \lesssim \frac{\chi_{I_0}(r)}{r^{\frac{n-2}{2}}}\frac{\chi_{I_j}(\lambda)}{\lambda^{\frac{n-1}{2}}}
    \,2^{-\frac{j}{2}} \big( \| \Phi \|_{L^\infty} + \|\Phi' \|_{L^\infty} \big)
    \lesssim2^{-j\frac{n}{2}},$$
as desired.

\subsubsection{Proof of \eqref{interpolation_value_2}}
To show \eqref{interpolation_value_2}, by H\"older's inequality,
it is enough to show that for $j,k\geq0$
$$\Big\| \chi_{I_k}(|x|) \int_{\mathbb{R}} \Big(e^{i(t-s)(-\Delta)^{a/2}}P_0^2 F_j(\cdot,s)\Big)(x)ds \Big\|_{L^2_{x,t}}
\lesssim2^{\frac{1}{2}(j+k)} \| F_j \|_{L^2_{x,t}}.$$
For this, we consider the operators $T_k$, $k\geq0$, defined for $r>0$ and $t\in\mathbb{R}$ by
    $$
    T_k h(r,t)=\chi_{I_k}(r) r^{-\frac{n-2}2} \int_0^\infty e^{it\rho^{a}} \overline{J_{\frac{n-2}2}(r\rho)} \varphi(\rho)h(\rho) d\rho,
    $$
where $\varphi(\rho)^2=\rho\psi(\rho)^2$.
Then the adjoint operator $T^*_{k}$ of $T_{k}$ is given for $\rho>0$ by
    $$
    T^{*}_{k} H (\rho)
    = \varphi (\rho) \int e^{-is\rho^{a}} \int \chi_{I_k}(\lambda)\lambda^{-\frac{n-2}{2}}
    J_{\frac{n-2}{2}} (\lambda\rho) H(\lambda,s) d\lambda ds,
    $$
and so
\begin{align*}
&T_kT^{*}_{j} (\lambda^{n-1}H) (r,t)\\
&\quad=\frac{\chi_{I_k}(r)}{r^{\frac{n-2}{2}}} \iint \frac{\chi_{I_j}(\lambda)}{\lambda^{\frac{n-2}{2}}}
\bigg( \int e^{i(t-s)\rho^{a}} \overline{J_{\frac{n-2}{2}} (r\rho)} J_{\frac{n-2}{2}} (\lambda\rho) \varphi^2 (\rho) d\rho \bigg)
\lambda^{n-1} H(\lambda,s) d\lambda ds.
\end{align*}
Then, by regarding $\rho\psi^2(\rho)$ and $F_j(\lambda y',s)$
as $\varphi^2(\rho)$ and $H(\lambda,s)$, respectively, in \eqref{0907},
Then, by we are reduced to showing that for $j,k\geq0$
$$\|T_kT^{*}_{j} (\lambda^{n-1}H)\|_{L_t^2\mathfrak{L}_r^2}
\lesssim2^{\frac{1}{2}(j+k)} \| H \|_{L_t^2\mathfrak{L}_r^2},$$
where $\mathfrak{L}_r^2 = L^2(r^{n-1}dr)$.
By the usual $TT^\ast$ argument, this follows from
    \begin{equation*}
    \| T_k h \|_{L_t^{2} \mathfrak{L}_r^{2}}\lesssim 2^{\frac{k}{2}} \| h \|_{L^2}
    \end{equation*}
for $k\geq0$.
To show this, by changing the variable $\rho$ to $\rho^{a}$, we first see that
    $$
    \int_0^\infty e^{it\rho^{a}} \overline{J_{\frac{n-2}2}(r\rho)} \varphi(\rho)h(\rho) d\rho
    =a^{-1}\int_0^\infty e^{it\rho} \overline{J_{\frac{n-2}2}(r\rho^{1/a})} \varphi(\rho^{1/a})h(\rho^{1/a}) \rho^{1/{a}-1}d\rho,
    $$
and so we get
\begin{align*}
\| T_k h \|_{L_t^2 \mathfrak{L}_r^2}
&= C\Big\| \chi_{I_k}(r) r^{-\frac{n-2}2} \overline{J_{\frac{n-2}2}(r\rho^{1/a})}
\varphi(\rho^{1/a})h(\rho^{1/a})\rho^{1/a-1}\Big\|_{\mathfrak{L}_r^2 L_\rho^2} \\
&=C\bigg(\int_{I_k}r^{-(n-2)}\int_{1/2}^2|J_{\frac{n-2}2}(r\rho)|^2|h(\rho)|^2\rho^{1-a}d\rho\,r^{n-1}dr \bigg)^{1/2}\\
 &\lesssim
    \bigg(\int_{I_k} r^{-(n-2)} \int_{1/2}^2 \min\big\{ (r\rho)^{n-2}, (r\rho)^{-1} \big\} |h(\rho)|^2 \rho^{1-a} d\rho\,r^{n-1}dr\bigg)^{1/2}\\
 &\lesssim 2^{k/2} \|h\|_{L^2},
\end{align*}
using Plancherel's theorem in $t$ and \eqref{Bessel_Grafakos}.


\subsection{Proof of \eqref{base_par}}\label{sec4.3}

Let us now show the inhomogeneous part \eqref{base_par} in Proposition \ref{prop_par}.
For this we show a stronger estimate
    \begin{equation}\label{eno}
    \bigg\|\int_{-\infty}^{t} e^{i(t-s)(-\Delta)^{a/2}} F(\cdot,s)ds\bigg\|_{L^2(w(x,t))}
    \leq C\|w\|_{\mathfrak{L}^{\beta,p}_{a-par}} \|F\|_{L^2(w(x,t)^{-1})}
    \end{equation}
which implies \eqref{base_par}.
Indeed, to deduce \eqref{base_par} from this,
first decompose the $L_t^2$ norm in the left-hand side of \eqref{base_par}
into two parts, $t\geq0$ and $t<0$. Then the latter can be reduced to the former
by a change of variables $t\mapsto-t$, and so we only need to consider the first part $t\geq0$.
But, since $[0,t)=(-\infty,t)\cap[0,\infty)$, by applying \eqref{eno} with $F$ replaced by $\chi_{[0,\infty)}(s)F$,
the first part follows directly, as desired.

To show \eqref{eno}, by duality we may show the following bilinear form estimate as before:
    \begin{equation*}
    \bigg|\bigg\langle\int_{-\infty}^{t} \Big(e^{i(t-s)(-\Delta)^{a/2}}P_0^2F(\cdot,s)\Big)(x)ds,G(x,t)\bigg\rangle\bigg|
    \leq C\|w\|_{\mathfrak{L}^{\beta,p}_{a-par}}\|F\|_{L^2(w^{-1})}\|G\|_{L^2(w^{-1})}.
    \end{equation*}
But, once we have Lemma \ref{local_prop}
replacing $\int_{\mathbb{R}}$ with $\int_{-\infty}^{t}$,
this estimate follows clearly by repeating the previous argument
used for the homogeneous part \eqref{freq_par}.
Since \eqref{reduce} is obviously valid for this replacement,
it does not affect the last two estimates in the lemma.
We only need to consider the first estimate \eqref{interpolation_value_2}.
For this we first modify it as
\begin{equation*}
\bigg|\bigg\langle\int_{-\infty}^{t} \Big(e^{i(t-s)(-\Delta)^{a/2}}P_0^2F_j(\cdot,s)\Big)(x)ds,G_k(x,t)\bigg\rangle\bigg|
\lesssim2^{(\frac{1}{2}+\varepsilon)(j+k)}\|F_j\|_{L^{2}_{x,t}}\|G_k\|_{L^{2}_{x,t}}
\end{equation*}
uniformly in $0<\varepsilon<1$.
Since $\varepsilon$ is arbitrary and may be sufficiently small,
it is not difficult to see that this modification is harmless in repeating the previous argument.
See Remark \ref{csreas} for the reason for this modification.

Now we show the above modified estimate.
Similarly as in \eqref{ell}, we may write
    \begin{align*}
    &\bigg\langle\int_{-\infty}^{t} \Big(e^{i(t-s)(-\Delta)^{a/2}}P_0^2F_j(\cdot,s)\Big)(x)ds, G_k(x,t)\bigg\rangle\\
    &=\sum_{\nu\in R\mathbb{Z}}
    \bigg\langle \int_{-\infty}^{t} \Big(e^{i(t-s)(-\Delta)^{a/2}}
    P_0^2\big([\phi_{\nu}^0+\sum_{l=1}^{\infty}(\phi_{\nu+}^l+\phi_{\nu-}^l)]F_j\big)(\cdot,s)\Big)(x)ds,\phi_\nu^0 G_k(x,t)\bigg\rangle.
    \end{align*}
As in \eqref{case:l>1}, we easily see that
for a sufficiently large number $N>0$,
    \begin{align}\label{sdga}
   \nonumber\sum_{\nu\in R\mathbb{Z}}\sum_{l=2}^{\infty} \bigg| \bigg\langle &\int_{-\infty}^{t} \Big(e^{i(t-s)(-\Delta)^{a/2}}P_0^2 (\phi_\nu^l F_j)(\cdot,s)\Big(x)ds, \phi_\nu^0 G_k(x,t) \bigg\rangle \bigg|\\
    \nonumber&\leq \sum_{\nu\in R\mathbb{Z}}\sum_{l=2}^{\infty} C_N (2^l R)^{-N}
        \big\| \phi_\nu^l F_j \big\|_{L^1_{x,t}} \big\| \phi_\nu^0 G_k \big\|_{L^1_{x,t}}\\
\nonumber&\lesssim\sum_{\nu\in R\mathbb{Z}}\sum_{l=2}^{\infty} C_N (2^l R)^{-N}2^{l/2}R2^{(j+k)n/2}
        \big\| \phi_\nu^l F_j \big\|_{L^2_{x,t}} \big\| \phi_\nu^0 G_k \big\|_{L^2_{x,t}}\\
    &\lesssim 2^{-\frac{N-n-1}{2}(j+k)}\sum_{l=2}^{\infty}2^{-(N-1)l}\sum_{\nu\in R\mathbb{Z}}
        \big\| \phi_\nu^l F_j \big\|_{L^2_{x,t}} \big\| \phi_\nu^0 G_k \big\|_{L^2_{x,t}},
 \end{align}
where $R=\max(2^j,2^k)\geq2^{(j+k)/2}$ and $\phi_\nu^l$ stands for $\phi_{\nu+}^l$ or $\phi_{\nu-}^l$.
Here we also used H\"older's inequality for the second inequality.
Next, using the Cauchy-Schwarz inequality together with the following trivial estimates
    \begin{equation*}
    \sum_{\nu \in R \mathbb{Z}} \big\| \phi_\nu^l F_j \big\|_{L^2}^2
    \leq C2^l \| F_j \|_{L^2}^2
\quad\text{and}\quad
    \sum_{\nu \in R \mathbb{Z}} \big\| \phi_\nu^0 G_k \big\|_{L^2}^2
    \leq C\| G_k \|_{L^2}^2,
    \end{equation*}
we bound
$$\sum_{\nu\in R\mathbb{Z}}
        \big\| \phi_\nu^l F_j \big\|_{L^2_{x,t}} \big\| \phi_\nu^0 G_k \big\|_{L^2_{x,t}}
\leq2^{l/2}\big\| F_j \big\|_{L^2_{x,t}} \big\| G_k \big\|_{L^2_{x,t}}.$$
Combining this and \eqref{sdga}, we conclude that
\begin{align*}
   \nonumber\sum_{\nu\in R\mathbb{Z}}\sum_{l=2}^{\infty} \bigg| \bigg\langle \int_{-\infty}^{t} \Big(e^{i(t-s)(-\Delta)^{a/2}}P_0^2 (\phi_\nu^l F_j)&(\cdot,s)\Big)(x)ds, \phi_\nu^0 G_k(x,t) \bigg\rangle \bigg|\\
    &\lesssim 2^{-\frac{N-n-1}{2}(j+k)}\|F_j\|_{L^2_{x,t}}\|G_k\|_{L^2_{x,t}}
 \end{align*}
for a sufficiently large number $N>0$.
Hence it suffices to show that
\begin{align}\label{inhomo_22_case_3}
\nonumber\sum_{\nu\in R\mathbb{Z}}
\bigg|\bigg\langle\int_{-\infty}^{t} \Big(e^{i(t-s)(-\Delta)^{a/2}}P_0^2(\phi_{\nu} F_j)&(\cdot,s)\Big)(x)ds,\phi_\nu^0 G_k(x,t)\bigg\rangle\bigg|\\
&\lesssim 2^{(\frac{1}{2}+\varepsilon)(j+k)}\|F_j\|_{L^2_{x,t}}\|G_k\|_{L^2_{x,t}},
\end{align}
where $\phi_{\nu}(t):=(\phi_{\nu}^0+\phi_{\nu+}^1+\phi_{\nu-}^1)(t)=\chi_{[v-2R, v+2R)}(t)$.
For this, we first observe that for $0<\varepsilon<1$,
\begin{align}\label{inhomo_22_case_2}
\nonumber\bigg|\bigg\langle\int_{\mathbb{R}} \Big(e^{i(t-s)(-\Delta)^{a/2}}
P_0^2(\phi_{\nu}F_j)&(\cdot,s)\Big)(x)ds, \phi_\nu^0G_k(x,t)\bigg\rangle\bigg|\\
&\lesssim 2^{(\frac{1}{2}-\frac n2(\frac{\varepsilon}{2-\varepsilon}))(j+k)}
\|\phi_{\nu}F_j\|_{L^{2-\epsilon}_{x,t}}\|\phi_{\nu}^0G_k\|_{L^{2-\epsilon}_{x,t}}
\end{align}
which follows from real interpolation between the estimates in Lemma \ref{local_prop}.
Indeed, we first note that from \eqref{interpolation_value_3} and \eqref{interpolation_value},
\begin{equation}\label{interp}
    \bigg| \bigg\langle \int_{\mathbb{R}} \Big(e^{i(t-s)(-\Delta)^{a/2}}P_0^2 F_j(\cdot,s)\Big)(x)ds, G_k(x,t) \bigg\rangle \bigg|
    \lesssim 2^{-\frac{n-1}{2}(j+k)} \|F_j\|_{L^{1}_{x,t}} \|G_k\|_{L^{1}_{x,t}}
    \end{equation}
for $j,k\geq0$, because
$2^{-\frac{2n-1}{4}(j+k)} 2^{-\frac{1}{4}|j-k|}\lesssim 2^{-\frac{n-1}{2}(j+k)}$
in \eqref{interpolation_value}.
Now, by real interpolation between \eqref{interpolation_value_2} and \eqref{interp},
with $\phi_{\nu}F_j$ and $\phi_\nu^0G_k$ instead of $F_j$ and $G_k$, respectively,
we get \eqref{inhomo_22_case_2} after some easy computations.

Using the dual characterisation of $L^p$ spaces and H\"older's inequality,
we also easily see that \eqref{inhomo_22_case_2} is equivalent to
\begin{equation}\label{krist}
\bigg\|\int_{\mathbb{R}}e^{i(t-s)(-\Delta)^{a/2}}
P_0^2(\phi_{\nu}F_j)(\cdot,s)ds\bigg\|_{L^{\frac{2-\varepsilon}{1-\varepsilon}}_{x,t}(I_k)}
\lesssim 2^{(\frac{1}{2}-\frac n2(\frac{\varepsilon}{2-\varepsilon}))(j+k)}
\|\phi_{\nu}F_j\|_{L^{2-\epsilon}_{x,t}}.
\end{equation}
Since $\frac{2-\varepsilon}{1-\varepsilon}>2-\varepsilon$,
we now get \eqref{krist} replacing $\int_{\mathbb{R}}$ with $\int_{-\infty}^t$,
which is equivalent to
\begin{align*}
\bigg|\bigg\langle\int_{-\infty}^t \Big(e^{i(t-s)(-\Delta)^{a/2}}
P_0^2(\phi_{\nu}F_j)&(\cdot,s)\Big)(x)ds, \phi_\nu^0G_k(x,t)\bigg\rangle\bigg|\\
&\lesssim 2^{(\frac{1}{2}-\frac n2(\frac{\varepsilon}{2-\varepsilon}))(j+k)}
\|\phi_{\nu}F_j\|_{L^{2-\epsilon}_{x,t}}\|\phi_{\nu}^0G_k\|_{L^{2-\epsilon}_{x,t}},
\end{align*}
by applying the following Christ-Kiselev lemma with $q=\frac{2-\varepsilon}{1-\varepsilon}$ and $p=2-\varepsilon$.

\begin{lem}[\cite{CK}]
Let $X$ and $Y$ be Banach spaces.
Assume that $T:L^p(\mathbb{R};X)\rightarrow L^q(\mathbb{R};Y)$, $1\leq p <q<\infty$, is a bounded linear operator
defined by
$$Tf(t)= \int_{\mathbb{R}} K(t,s) f(s)ds,$$
where $K:\mathbb{R}\times\mathbb{R}\rightarrow B(X,Y)$ and $B(X,Y)$ is the space of bounded linear transformations from $X$ to $Y$.
Then the operator $T$ replacing $\int_{\mathbb{R}}$ with $\int_{-\infty}^t$
has the same $L^p-L^q$ boundedness.
\end{lem}

\begin{rem}\label{csreas}
This lemma does not hold when $q=p$. So, if we consider directly the estimate without the modification, that is, with $\varepsilon=0$,
the exponents $q=\frac{2-\varepsilon}{1-\varepsilon}$ and $p=2-\varepsilon$ in \eqref{krist} are equivalent each other.
Notice that the lemma does not work for this case.
\end{rem}

Then by H\"older's inequality,
    \begin{align*}
    \bigg| \bigg\langle \int_{-\infty}^{t}& \Big(e^{i(t-s)(-\Delta)^{a/2}}P_0^2 (\phi_{\nu} F_j)(\cdot,s)\Big)(x)ds, \phi_\nu^0 G_k(x,t) \bigg\rangle \bigg|\\
    &\lesssim2^{(\frac{1}{2}-\frac n2(\frac{\varepsilon}{2-\varepsilon}))(j+k)}
    (R2^{jn})^{\frac{\varepsilon}{2(2-\varepsilon)}} (R2^{kn})^{\frac{\varepsilon}{2(2-\varepsilon)}}
    \|\phi_{\nu}F_j\|_{L^{2}_{x,t}} \| \phi_{\nu}^0 G_k\|_{L^{2}_{x,t}}\\
    &\lesssim 2^{(\frac{1}{2} +\epsilon) (j+k)}
        \|\phi_{\nu}F_j\|_{L^{2}_{x,t}} \| \phi_{\nu}^0 G_k\|_{L^{2}_{x,t}}.
    \end{align*}
Here, for the last inequality, we have used the fact that $R=\max(2^j,2^k)\leq2^{j+k}$.
By summing in $\nu$ and using the Cauchy-Schwarz inequality as before,
we get the desired estimate \eqref{inhomo_22_case_3}.


\section{Proof of Theorem \ref{thm0}}\label{sec5}

In this final section, we deduce the well-posedness (Theorem \ref{thm0})
for the Cauchy problem \eqref{higher} from the weighted $L^2$ Strichartz estimates in Theorem \ref{thm_par}
using the fixed point argument.

The starting point is that the solution of \eqref{higher} can be given by the following integral equation
\begin{equation}\label{sol}
u(x,t)= e^{it(-\Delta)^{a/2}}u_0(x) - i\int_0^t \Big(e^{i(t-s)(-\Delta)^{a/2}} F(\cdot,s)\Big)(x)ds
+\Phi(u)(x,t),
\end{equation}
where
$$\Phi(u)(x,t)=-i\int_0^t \Big(e^{i(t-s)(-\Delta)^{a/2}}(Vu)(\cdot,s)\Big)(x)ds.$$
Here we observe that
$$(I-\Phi)(u)=e^{it(-\Delta)^{a/2}}u_0 - i\int_0^t e^{i(t-s)(-\Delta)^{a/2}} F(\cdot,s)ds,$$
where $I$ is the identity operator.
Then, since $u_0\in L^2$ and $F\in L^2(|V|^{-1})$,
by applying the weighted $L^2$ Strichartz estimates in Theorem \ref{thm_par} with $w=|V|$,
we see that
$$(I-\Phi)(u)\in L^2(|V|).$$
Hence, it is enough to show that
the operator $I-\Phi$ has an inverse in the space $L^2(|V|)$, needed for the fixed point argument.
For this, we want to show that the operator norm for $\Phi$ in the space $L^2(|V|)$
is strictly less than $1$.
Namely, we show that $\|\Phi(u)\|_{L^2(|V|)}<\frac12\|u\|_{L^2(|V|)}$.
Indeed, from the inhomogeneous estimate \eqref{inho_par} with $w=|V|$, it follows that
\begin{align}\label{inv}
\nonumber\|\Phi(u)\|_{L^2(|V|)}&\leq C\|V\|_{\mathfrak{L}_{a-par}^{a,p}}\|Vu\|_{L^2(|V|^{-1})}\\
\nonumber&=C\|V\|_{\mathfrak{L}_{a-par}^{a,p}}\|u\|_{L^2(|V|)}\\
&<\frac12\|u\|_{L^2(|V|)}.
\end{align}
Here, for the last inequality, we have used the smallness assumption on the norm $\|V\|_{\mathfrak{L}_{a-par}^{a,p}}$.

On the other hand, from \eqref{sol}, \eqref{inv} and Theorem \ref{thm_par}, we easily see that
\begin{align}\label{1123}
\nonumber\|u\|_{L^2(|V|)}&\leq C\big\|e^{it(-\Delta)^{a/2}}u_0 \big\|_{L^2(|V|)}
+C\bigg\|\int_{0}^{t}e^{i(t-s)(-\Delta)^{a/2}}F(\cdot,s)ds\bigg\|_{L^2(|V|)}\\
&\leq C\|V\|_{\mathfrak{L}_{a-par}^{a,p}}^{1/2}\|u_0\|_{L^2}
+C\|V\|_{\mathfrak{L}_{a-par}^{a,p}}\|F\|_{L_{t,x}^2(|V|^{-1})}.
\end{align}
Now \eqref{1-1} is proved.
To show \eqref{1-2}, we will use \eqref{1123} and the following estimate
\begin{equation}\label{dual}
    \bigg\|\int_{-\infty}^{\infty}e^{-is(-\Delta)^{a/2}}F(\cdot,s)ds\bigg\|_{L_x^2}
    \leq C\|w\|_{\mathfrak{L}_{a-par}^{a,p}}^{1/2}\|F\|_{L^2(w(x,t)^{-1})}
\end{equation}
which is just the dual estimate of \eqref{hop_H_par}.
First, from \eqref{sol}, \eqref{dual} with $w=|V|$, and the simple fact that $e^{it(-\Delta)^{a/2}}$ is an isometry in $L^2$,
it follows that
$$\|u\|_{L_x^2}\leq C\|u_0\|_{L^2}
+C\|V\|_{\mathfrak{L}_{a-par}^{a,p}}^{1/2}\|F\|_{L^2(|V|^{-1})}
+C\|V\|_{\mathfrak{L}_{a-par}^{a,p}}^{1/2}\|Vu\|_{L^2(|V|^{-1})}.$$
Since $\|Vu\|_{L^2(|V|^{-1})}=\|u\|_{L^2(|V|)}$
and $\|V\|_{\mathfrak{L}_{a-par}^{a,p}}$ is small enough,
from this and \eqref{1123}, we now get
$$\|u\|_{L_x^2}\leq C\|u_0\|_{L^2}+C\|V\|_{\mathfrak{L}_{a-par}^{a,p}}^{1/2}\|F\|_{L_{t,x}^2(|V|^{-1})}$$
as desired.
This completes the proof.

\

\noindent\textbf{Acknowledgments.}
The authors thank the anonymous referees for many valuable suggestions which improve our presentation a great deal.



\end{document}